\documentclass[11pt, a4paper, oneside]{article}
\usepackage{tikz}
\usepackage{amsmath}
\usepackage{amsthm}
\usepackage{amssymb}
\usepackage{hyperref,soul,color}
\usepackage{graphicx}
\usepackage{comment}
\usepackage{colortbl}
\usepackage{placeins}
\usepackage{xcolor,soul}
\usepackage{mathtools}

\usepackage{pgfplots}

\newcommand{\cG}{\mathcal{G}}
\newcommand{\cI}{\mathcal{I}}
\newcommand{\cS}{\mathcal{S}}
\newcommand{\cW}{\mathcal{W}}

\newcommand{\de}{\mathrm{d}}

\newcommand{\abs}[1]{\lvert#1\rvert}
 
\pgfplotsset{compat=newest}

\usepackage{graphicx} 

\newtheorem{theorem}{Theorem}[section]
\newtheorem{proposition}[theorem]{Proposition}

\theoremstyle{remark}
\newtheorem{remark}[theorem]{Remark}
\theoremstyle{definition}

\title{
Equilibrium and Interaction Regimes in Mixed Disclination-Dislocation Systems
}

\author{Pierluigi Cesana\footnote{Institute of Mathematics for Industry, Kyushu University, Japan, \url{cesana@math.kyushu-u.ac.jp}
},  
 Marco Morandotti\footnote{
Dipartimento di Scienze Matematiche ``G.~L.~Lagrange'', Politecnico di Torino, Italy, \url{marco.morandotti@polito.it}
},
Aldo Sambo
\footnote{
Dipartimento di Matematica, Informatica e Geoscienze, University of Trieste, Italy, \url{s322081@ds.units.it}
}}

\begin{document}

\maketitle

\begin{abstract}

We investigate equilibrium configurations and dissipative dynamics of simplified disclination--dislocation systems in planar elasticity. We consider an edge dislocation interacting with a fixed wedge disclination in a circular domain. We derive the reduced interaction energy, characterize the equilibrium states and their stability, and study the associated dissipative evolution. The analysis reveals characteristic length scales and collision times arising from the competition between the dislocation and disclination contributions. In particular, the presence of the disclination produces a nontrivial modification of the dislocation dynamics and can hinder its motion. This behavior is consistent with experimental observations and provides a simple mathematical description of the interaction between rotational and translational defects.

\end{abstract}

\vskip5pt
\noindent
\textsc{Keywords}: Wedge Disclinations, Edge Dislocations, Linearized Elasticity, Dynamics of Defects, Interaction Length Scales.
\vskip5pt
\noindent
\textsc{2020 AMS subject classification:}  
49J45,   
49J10,   
74B15,  
74H60,
74H80.

   \tableofcontents

\section{Introduction}

 Defects, lattice asymmetries, vacancies, and other microstructural irregularities play a fundamental role in determining the mechanical behavior of crystalline solids. Among the most important lattice incompatibilities are \emph{dislocations} and \emph{disclinations}, which describe two distinct types of kinematic incompatibilities in the crystal structure. Dislocations are translational defects characterized by the failure of closure of the translational component of the displacement field. They are described by their \emph{Burgers vector}, which measures the lattice mismatch at the microscopic scale. Disclinations, on the other hand, are rotational incompatibilities characterized by the failure of closure of the rotational component of the displacement field. They are described by their \emph{Frank angle}, which quantifies the rotational mismatch induced in the crystal lattice. Although dislocations and disclinations have been extensively studied individually, their mutual interaction remains much less understood, despite its importance in many physical phenomena and engineering applications.

Since the pioneering work of Volterra~\cite{V07}, dislocations and disclinations have been studied within continuum elasticity; see, for example, \cite{W68-BIS,N67,RV1983}. Dislocations were identified as a fundamental mechanism underlying plastic deformation in the seminal works \cite{Orowan1934,Polanyi1934,Taylor1934} and have subsequently attracted substantial mathematical interest, in particular because of their close connection with variational models for Ginzburg--Landau vortices~\cite{BBH1994}.

The interaction between dislocations and disclinations is supported by a broad range of experimental observations. In crystal plasticity, their coupling is closely related to the formation and evolution of kinks and kink bands \cite{HAGIHARA10,I19,IS20,LN15}. Mobile partial or effective disclinations have also been proposed as carriers of rotational plasticity during severe deformation, including nanograin fragmentation in mechanically milled Fe and kink-band evolution in long-period stacking ordered Mg alloys \cite{murayama02,TOKUZUMI2023118785,TOKUZUMI2020100716}. In shape-memory alloys, experiments on $\beta$-Ti martensites reveal disclination-type incompatibilities localized at martensitic variant junctions \cite{IHM13,ILTH17}. Related structures arise in graphene, where combinations of disclinations generate dipoles, quadrupoles, Stone--Wales defects, flower defects, and edge dislocations represented as disclination dipoles \cite{Banhart11}. Their role in determining elastic fields, energetics, grain boundaries, and other crystalline structures has been investigated in \cite{met10111517,klem}.

Several mathematical approaches have been developed to describe systems involving both dislocations and disclinations. A recent formulation based on Lie algebra techniques derives coupled governing equations for translational and rotational incompatibilities within a unified geometric setting \cite{DU2025176}. Equilibrium configurations involving disclination dipoles and edge dislocations in layered elastic materials have been studied in \cite{Colin2025Disclination}. Although that work is also formulated within linear elasticity, it concerns the interaction between a dislocation and a \emph{disclination dipole}; as will become clear below, the interaction of an isolated wedge disclination with an edge dislocation leads to a different mathematical problem. Another important framework is the generalized disclination theory introduced in \cite{acharya15}, which treats Volterra-type incompatibilities within a continuum theory of distributed fields and has been applied to disclinations, dislocations, grain boundaries, and their interactions \cite{ZHANG18}.
Related stochastic approaches to the evolution of interfaces in self-similar martensitic microstructures, where lattice incompatibilities play a central role, can be found in \cite{BCH15,CH20}.

The dynamics of dislocations has also been studied extensively, especially under dissipative evolution laws; see \cite{Acharya10,ACHARYA20171,tng}. The configurational forces governing their motion, including the Peach--Koehler force derived from the Eshelby stress tensor, were analyzed in \cite{CermelliLeoni06}; see also \cite{CermelliGurtin99} and \cite{BlassFonsecaLeoniMorandotti15,BlassMorandotti17,HudsonMorandotti2017} for related results in the case of screw dislocations. We also recall the continuum description based on continuously distributed dislocations developed in \cite{ACHARYA01}. By contrast, the dissipative dynamics of disclinations is much less developed. A first quantitative study for finite systems of wedge disclinations was carried out in \cite{CGMP2025}, where characteristic collision times, interactions between disclinations of opposite charges, and interactions with the boundary were analyzed. Mesoscale reaction--diffusion models describing the nucleation and coupled evolution of dislocations and disclinations were proposed in \cite{ROMANOV1993707,ROMANOV19941581}. These models include empirical terms accounting for nucleation, motion, and mutual interaction, but do not address the interaction between an isolated Volterra disclination and an edge dislocation through a variational description.

A basic difficulty in constructing such a model is the different energetic character of the two incompatibilities. Within planar linear elasticity, wedge disclinations have finite, although very large, elastic energy scaling with the square of the size of the body; see, e.g., \cite{CDLM24,SN88}. Edge dislocations, instead, have logarithmically divergent self-energy; see, e.g., \cite{HirthLothe82,N67}. A systematic variational treatment of isolated dislocations and disclinations has recently been developed in \cite{CDLM24,CFM2025}, where both arise from linearized elasticity with prescribed translational and rotational incompatibilities and the corresponding length scales are clarified. The different energy scalings show, however, that a direct superposition of the two elastic fields does not by itself lead to a meaningful interaction regime.

Various regularization procedures have been proposed to handle the singular character of these fields. These include core-radius regularizations for dislocations and disclinations \cite{CermelliLeoni06,CDLM24,CFM2025}, as well as approaches based on first-gradient elasticity \cite{LAZAR05} and second strain-gradient elasticity \cite{DENG20073646, lazar06}. In the present work we follow instead the classical construction of Eshelby~\cite{Eshelby66}, in which an edge dislocation is represented as the limit of a dipole of wedge disclinations. More precisely, the distortion generated by an edge dislocation is approximated by two disclinations of opposite Frank angles separated by a distance $h$, with the Frank angles, separation distance, and Burgers vector linked through the classical dipole relation. This allows the two contributions to be treated within a common variational framework and provides a consistent asymptotic description of their interaction.

Here we apply this construction to a simplified planar configuration that isolates the main mechanisms of disclination--dislocation interaction. We consider a circular domain containing a wedge disclination fixed at its center, while an edge dislocation is constrained to move radially and its Burgers vector is free to rotate; see Figure~\ref{figure1}. The reduced problem therefore has two degrees of freedom, one spatial and one angular. Despite this low-dimensional structure, the resulting energy landscape exhibits multiple equilibrium states and nontrivial stability and dynamical behavior.

Our analysis identifies a regime, determined by a dimensionless parameter $\alpha$ measuring the relative strength of the disclination and dislocation contributions, in which stable equilibria occur in the interior of the domain. The interaction also generates characteristic length scales arising from the competition between the corresponding elastic and configurational contributions.

We then study the associated dissipative dynamics and identify two qualitatively different behaviors. Some trajectories lead to collision of the edge dislocation with the boundary in finite time, as occurs for isolated edge dislocations, whereas trajectories attracted to stable interior equilibria approach them only asymptotically as time tends to infinity. In particular, the presence of the disclination can drastically modify the dynamics of the dislocation: its elastic stress field can significantly slow down, and eventually arrest, the radial motion of the dislocation at a stable interior equilibrium.

This last result provides a direct connection with experimental observations. In deformation experiments on Mg-based alloys, dislocation motion has been observed to be hindered by the stress fields generated by fixed disclinations \cite{TOKUZUMI2023118785}. Remarkably, this experimentally observed behavior emerges naturally from the present variational model: the interaction with a single fixed disclination is sufficient to slow down and, in the appropriate parameter regime, arrest the motion of an edge dislocation.

\begin{figure}[ht]
    \centering
\begin{tikzpicture}
    \draw[->] (4.5, 0) -- (11.5, 0) node[right] {$x_1$};
    \draw[->] (8, -3.5) -- (8, 3.5) node[above] {$x_2$};
    \draw (8,0) circle (3);
    \node at (6,2) [below right] {$B_1$};
    \fill (9.5, 0.5) circle (1.5pt) node[above left] {$(\xi_1,s_1)$};
    \fill (10, 0) circle (1pt) node[below left] {$d$};
    \fill (10.5, -0.5) circle (1.5pt) node[below right] {$(\xi_2,s_2)$};
    \fill (8, 0) circle (1.5pt) node[below left] {$(\xi_3,s_3)$};
     \draw[->] (10.5, -0.5) -- (9.5, 0.5);
    \draw[->] (10.4, 0) arc (0:135:0.4);
      \node at (10.5, 0.25) [above] {$\varphi$};
 \end{tikzpicture}
\caption{A dipole of disclinations $(\xi_1,s_1)$-$(\xi_2,s_2)$, centered at $d=(z,0)$ and with dipole length $h>0$, and an isolated disclination $(\xi_3,s_3)$. The dipole director~$w$ points from~$\xi_2$ to~$\xi_1$ and is oriented with an angle~$\varphi$ with respect to the positive $x_1$-axis. 
The isolated disclination $(\xi_3,s_3)$ is hinged at the center of the disk ($\xi_3=(0,0)$). 
The dipole can rotate and its center slide along the $x_1$-axis; the isolated disclination cannot move from the center.}
\label{figure1}
\end{figure}
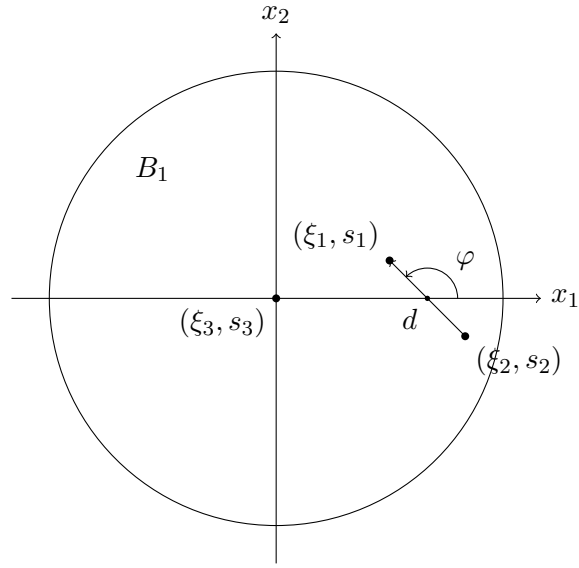

\section{Mechanical model}\label{sec_model}
To model the presence of $K\in\mathbb{N}$ isolated discinations in the unit disk $B_1$ of $\mathbb{R}^2$, we consider the concentrated measure $\theta\in H^{-2}(B_1)$   given by the superposition of Dirac deltas 
\begin{equation}\label{atomic_measure}
\theta=\sum_{k=1}^K s_k\delta_{\xi_k}\,, \qquad\xi_k\in B_1\,,\qquad \xi_k\neq\xi_j \quad\text{if} \quad k\neq j,
\end{equation}
where $s_k$'s are the associated Frank angles.
The corresponding functional $\cI^\theta\colon H^2_0(B_1)\to\mathbb{R}$ is given by the sum of the mechanical energy $\cW \colon H^2_0(B_1)\to[0,+\infty)$ and of the action of the measure $\theta$, and reads, in the formulation for the Airy potential~$v$, see \cite[formula~(0.2)]{CDLM24}, 
\begin{equation}\label{eq_en_B_1}
     \mathcal{I}^{\theta}(v)= \cW(v)+\langle \theta, v\rangle =\frac{1}{2}\frac{1+\nu}{E}\int_{B_1} \big[|\nabla^2 v(x)|^2-\nu(\Delta v(x))^2 \big]\,\de x+\langle\theta,v \rangle.
\end{equation}
We use the notation $H^2_0(B_1)$ and $H^{-2}(B_1)$ to denote the usual Sobolev space of functions that are square integrable together with their first- and second-order partial derivatives and such that their trace and that of their normal derivative on $\partial B_1$ vanishes, and its dual space, respectively.

When treating an arrangement of disclinations that includes dipoles $\xi^{h}_{\pm}=d\pm\frac{h}2w_\varphi$ centered in $d\in B_1$\, of length $h>0$, and oriented in the direction $w_\varphi=(\cos\varphi,\sin\varphi)$, we trace the dependence on $h$ by denoting by $\theta^h$ the corresponding measure, which will contain a contribution $s(\delta_{\xi^{h}_{+}}-\delta_{\xi^{h}_{-}})$.
More precisely, if there is just one dipole in the system, we have $\theta^h=s\big(\delta_{d+\frac{h}2w_\varphi}-\delta_{d-\frac{h}2w_\varphi}\big)$, that is, by a slight abuse of notation,
\begin{equation}
\theta^h(x)=s\bigg(\delta\Big(x-d-\frac{h}2w_\varphi\Big)-\delta\Big(x-d+\frac{h}2w_\varphi\Big)\bigg),
\end{equation}
where $\pm s$ are the Frank angles of the disclinations at $\xi^{h}_{\pm}$\,.

By linearly rescaling the Airy potential with respect to the dipole length $h$, that is 
\begin{equation}\label{2509140958}
v^{h}=hv
\end{equation}
for a certain function $v\in H^2_0(B_1)$, we obtain that 
\begin{equation}\label{202509121117_bis_1}
\cI^{\theta^h}(hv)= h^2\cI^{\theta^h/h}(v).
\end{equation}
This is motivated by \cite[Proposition~3.5]{CDLM24}, where it was shown that, as $h\to0$, a dipole of disclinations is energetically equivalent to an edge dislocation.

The Euler--Lagrange equation associated with the functional in the right-hand side of \eqref{202509121117_bis_1} is
\begin{equation}\label{2509121502}
\begin{cases}
\displaystyle\frac{1-\nu^2}{E} \Delta^2 v = -\frac{\theta^{h}}{h} & \text{in $B_1$\,,}\\[2mm]
v = \partial_n v = 0 & \text{on } \partial B_1\,.
\end{cases}
\end{equation}
If the measure $\theta^{h}$ is atomic as in \eqref{atomic_measure}, then the unique solution $\underline{v}\in H^2_0(B_1)$ to \eqref{2509121502} can be obtained via superposition and convolution with the Green's function (see formula \eqref{2509121514_app} in Appendix~\ref{ssec_min_en} with $R=1$)
\begin{equation}\label{2509121514}
\underline{v}(x)= -\bigg(\frac{\theta_{h}}{h}* G_{\xi}\bigg)(x)=\frac{C}{h}
\sum_{k=1}^K s_k\, \overline{G}_{\xi^h_{k}}(x),
\end{equation}
where $C=\frac{E}{1-\nu^2}\frac{1}{16\pi}$ (see \eqref{2510290905} below). 
For a derivation of the explicit expression of $\underline{v}$, of the Green's function $G_\xi$\,, and of the functions $\overline{G}_{\xi^h_{k}}$\,, we refer the reader to Sections~\ref{ssec_Green} and~\ref{ssec_min_en} in the Appendix, where the general derivation is carried out in the disk $B_R$\,.
Thanks to Clapeyron's Theorem (see, \emph{e.g.}, \cite[Theorem 2.2]{CGMP2025}), we can now compute
the mechanical energy at equilibrium (compare with \eqref{2510291059_app} below with $R=1$) 
\begin{equation}\label{2510291059}
\begin{split}
\underline{W}^{h} \coloneqq \cW(\underline{v})  = & -\frac{1}{2} \biggl\langle -\frac{\theta_{h}}{h}, \underline{v} \biggr\rangle 
= \frac{C}{2}
\cS^{h}\circ \overline{\cG}^{h}\,,
\quad\quad\text{where}\quad
\cS^{h}_{kj}=\frac{s_k s_j}{h^2},\; \overline{\cG}^{h}_{kj}=\overline{G}_{\xi^h_{j}}(\xi^h_{k}),
\end{split} 
\end{equation}  
and $\circ$ denotes the Hadamard (\emph{i.e.}, element-wise) product of matrices.
Observe that the matrix $\overline{\cG}$ is symmetric owing to the symmetry of the function $\overline{G}$, \emph{i.e.}, $\overline{G}_{\xi}(x)=\overline{G}_{x}(\xi)$\,, see \eqref{2510290857}.

\smallskip

From now on, we will consider  a set of $K=3$ disclinations, $\xi^h_{1}\,,\xi^h_{2}\,,\xi_3\in B_1$\,,
so that the corresponding measure is $\theta^h= s_1\delta_{\xi^h_{1}}+s_2\delta_{\xi^h_{2}}+s_3\delta_{\xi_3}$\,.
Here, the pair $\xi^h_{1}$-$\xi^h_{2}$ is a dipole as $\xi^h_{\pm}$ described above, more precisely
\begin{equation}\label{2confprimi}
\xi^h_{1} = (z,0)+\frac{h}{2}( \cos \varphi, \sin \varphi ),
\qquad
\xi^h_{2} = (z,0) -\frac{h}{2}( \cos \phi, \sin \phi ) ,
\qquad 
\xi_3 = ( 0, 0 )
\end{equation}
for $z\in (\frac{h}2,1)$, and our choice for the Frank angles is
\begin{equation}\label{Frank_angles}
s_1=-s_2\eqqcolon s>0\,, \qquad\text{and}\qquad s_3=h\sigma>0\,.
\end{equation}
\begin{remark}
A key ingredient of our analysis is the identification of the appropriate scaling regime governing the defects interaction. 
In view of considering the (rescaled) measure $\theta^h/h$, it is reasonable to assume that the Frank angle of the isolated wedge disclination scales as $h^\gamma$, for some exponent $\gamma>0$, which crucially determines the behavior of the coupled system in the limit as $h\to0$.
When $\gamma=1$, as in our choice~\eqref{Frank_angles}, the energetic contributions of the limit dislocation and the disclination remain of the same order, producing a genuine interaction regime. For $0<\gamma<1$, the disclination dominates the asymptotic behavior, whereas for $\gamma>1$ the influence of the isolated disclination disappears in the limit and the system reduces to that of an isolated dislocation. This critical scaling coincides with the variational regime identified in~\cite{CDLM24}, where the equivalence between an edge dislocation and a dipole of wedge disclinations was established through a suitable renormalization procedure.
\end{remark}

Owing to~\eqref{Frank_angles}, the matrices $\cS^{h}$ and $\overline{\cG}^{h}$ in~\eqref{2510291059} read
\begin{equation}\label{SG}
\cS^{h} = 
\begin{pmatrix}
\quad s^2/h^2 & -s^2/h^2 & \quad s\sigma/h \\
-s^2/h^2 & \quad s^2/h^2 & -s \sigma/h \\
\quad\!\!\! s \sigma/h & -s \sigma/h & \quad \sigma^2
\end{pmatrix}
\quad \text{and}\quad
\overline{\cG}^{h} =
\begin{pmatrix}
\overline{G}_{\xi^h_{1}}(\xi^h_{1}) & \overline{G}_{\xi^h_{2}}(\xi^h_{1}) & \overline{G}_{\xi_{3}}(\xi^h_{1}) \\
\overline{G}_{\xi^h_{1}}(\xi^h_{2}) & \overline{G}_{\xi^h_{2}}(\xi^h_{2}) & \overline{G}_{\xi_{3}}(\xi^h_{2}) \\
\overline{G}_{\xi^h_{1}}(\xi_{3}) & \overline{G}_{\xi^h_{2}}(\xi_{3}) & \overline{G}_{\xi_{3}}(\xi_{3})
\end{pmatrix},
\end{equation}
and in the expression of the mechanical energy at equilibrium~$\underline{W}^h$ in~\eqref{2510291059} we can identify the sum of three different contributions 
\begin{equation}\label{2511091901}
\underline{W}^h(\varphi,z;s,\sigma)=\underline{W}_{\text{D}}^h(\varphi,z;s)+\underline{W}_{\text{int}}^h(\varphi,z;s,\sigma)+\underline{W}_{\text{d}}(\sigma),
\end{equation}
where
\begin{eqnarray}
&&\!\!\!\! \underline{W}^h_{\text{D}}(\varphi,z;s)= \frac{C}{2h^2} \sum_{k,j=1}^2 s_k s_j \,\overline{G}_{\xi^h_{j}}(\xi^h_{k}) = \underline{W}_{\text{D,eff}}^h(\varphi,z;s)+\underline{W}_{\text{D,div}}^h(s)\label{W1h_2conf}\\
&\!\!\!\! \coloneqq &\!\!\!\! Cs^2\biggl(2z^2\cos^2\varphi+\log\biggl((1-z^2)^2+\frac{h^2}{16}(8+h^2-8z^2\cos(2\varphi))\biggr)\!\! \biggr)-2Cs^2\log h, \nonumber
\end{eqnarray}
\begin{eqnarray}
&&\!\!\!\! \underline{W}_{\text{int}}^h(\varphi,z;s,\sigma) = \frac{C\sigma}{h} \sum_{k=1}^2 s_k\,\overline{G}_{\xi_{3}}(\xi^h_{k}) \label{223223}\\
&\!\!\!\!=&\!\!\!\!  \frac{Cs\sigma}{2h} \bigg[(h^2{+}4z^2)\tanh^{-1}\!\bigg(\frac{4hz\cos\varphi}{h^2+4z^2}\bigg)
{+}2hz\cos\varphi\bigg(\!\log\bigg(z^4{+}\frac{h^2}{16}\big(h^2-8 z^2 \cos(2\varphi)\big)\!\!\bigg)-2\bigg)\bigg], \nonumber
\end{eqnarray}
and $\underline{W}_{\text{d}}(\sigma)= \frac12 C\sigma^2\, \overline{G}_{\xi_{3}}(\xi_{3})=C\sigma^2/2$.
The three summands in~\eqref{2511091901} represent: the contribution to the energy of the disclinations in the dipole, the interaction of the disclinations in the dipole with the isolated disclination, and the contribution to the energy of the isoleted disclination, respectively. 
Note that we made explicit use of the symmetry of $\cS^{h}$ and $\overline{\cG}^{h}$ in the expression for~$\underline{W}_{\text{int}}^h$\,.
Also notice that the orientation $\varphi$ of the dipole impacts both the energy $\underline{W}_{\text{D}}^h$ of the dipole (which now is not symmetric with respect to the origin) and the interaction energy $\underline{W}_{\text{int}}^h$\,.

By the explicit expression in \eqref{W1h_2conf}, we notice that  
\begin{equation}\label{W_1_lim}
\underline{W}_{\text{D,div}}^h(s)=-\frac{E}{1-\nu^2}\frac{s^2}{8\pi}\log h \to +\infty,\qquad\text{as $h\to0^+$,}
\end{equation}
in accordance with \cite[formula~(3.9) and Theorem~4.6]{CDLM24} and \cite[formula~(2.18), with $R=1$ and $a=h$]{SN88} (compare also with \cite[formula (7.17)]{dW3} and \cite[formula (7)]{ZA2018} for the relationship between the Frank angle $s$ of the dipole and the Burgers vector $b$ of the resulting dislocation).
Given the behavior \eqref{W_1_lim}, it is customary (see, \emph{e.g.}, \cite{AlicandroDeLucaGarroniPonsiglione16, BlassFonsecaLeoniMorandotti15, BlassMorandotti17, CermelliLeoni06}) to disregard the diverging term and to work with $\underline{W}^h_{\text{eff}}$\,, defined by
\begin{equation}\label{W_eff_h_2conf}
\underline{W}_{\text{eff}}^h(\varphi,z;s,\sigma)\coloneqq \underline{W}^h(\varphi,z;s,\sigma)-\underline{W}_{\text{D,div}}^h(s),
\end{equation}
which is often referred to as the \emph{renormalized energy}.

We compute the torque and force exerted on the dipole
\begin{equation} 
-\begin{pmatrix}
    t^h(\varphi,z;s,\sigma)\\
    f^h(\varphi,z;s,\sigma)
\end{pmatrix}\coloneqq
\nabla_{(\varphi,z)} \underline{W}_{\text{eff}}^h (\varphi,z;s,\sigma) = 
\begin{pmatrix}
    Cs z\,w^h(\varphi,z;s,\sigma)\sin\varphi \\
    Cs \bigl(F_1^h+F_2^h+F_3^h\bigr)/h
\end{pmatrix},
\end{equation}
where 
\begin{equation*}
\begin{split}
w^h(\varphi,z;s,\sigma) \coloneqq &\, -4 s z \cos\varphi \bigg(\frac{h^4 + 16(1 - z^2)^2 -8h^2 z^2 \cos(2\varphi)}{8 h^2 + h^4 + 16 (1 - z^2)^2 -8h^2 z^2 \cos(2\varphi)}\bigg) \\
&\, -\sigma\log\bigg(\frac{h^4+16z^4 - 8 h^2 z^2 \cos(2\varphi)}{16}\bigg),
\end{split}
\end{equation*}
and 
\begin{eqnarray*}
F_1^h &\coloneqq&  4\sigma z \tanh^{-1}\bigg(\frac{4hz\cos\varphi}{h^2+4z^2}\bigg), \qquad F_3^h\coloneqq \sigma h\cos\varphi\, \log\bigg(\frac{h^4+16z^4 - 8h^2 z^2 \cos(2\varphi)}{16}\bigg)\\
F_2^h &\coloneqq& \frac{2sh}{z}\bigg(1+z^2+z^2\cos(2\varphi)+\frac{16z^4-(4+h^2)^2}{8 h^2 + h^4 + 16 (1 - z^2)^2 -8h^2 z^2 \cos(2\varphi)}\bigg)
\end{eqnarray*}
The following commutativity result holds true.
\begin{proposition}\label{prop_comm_2conf}
Let $\underline{W}_{\mathrm{eff}}^h$ be defined as in~\eqref{W_eff_h_2conf}. 
Then the gradient operation commutes with that of taking the limit as $h\to0$.
By defining $\underline{W}_{\mathrm{eff}}(\varphi,z;s,\sigma) \coloneqq \lim_{h\to0}  \underline{W}_{\mathrm{eff}}^h(\varphi,z;s,\sigma)$, we have that
\begin{equation}\label{commuted_2conf}
\begin{split}
\nabla_{(\varphi,z)}\underline{W}_{\mathrm{eff}}(\varphi,z;s,\sigma)= &\, \lim_{h\to0} \nabla_{(\varphi,z)} \underline{W}_{\mathrm{eff}}^h(\varphi,z;s,\sigma) \\
=&\, 4Cs\begin{pmatrix}
            -z\bigl( sz\cos\varphi+\sigma \log z\bigr) \sin\varphi \\
            sz\bigg(\cos^2\varphi -\frac{1}{1-z^2}\biggr)+\sigma(\log z+1)\cos\varphi
        \end{pmatrix}
\end{split}
\end{equation}
\end{proposition}
\begin{proof}
The continuity of $w^h$, $F^h_1$\,, $F^h_2$\,, and $F^h_3$ with respect to $h$ yields that
\begin{equation}\label{eq_Weff_2conf}
\begin{split}
\underline{W}_{\text{eff}}(\varphi,z;s,\sigma) = &\, \lim_{h\to0}  \underline{W}_{\text{eff}}^h(\varphi,z;s,\sigma) \\
= &\, 2Cs^2\bigl( z^2\cos^2\varphi+\log(1-z^2) \bigr)+ 4Cs\sigma z\log z\cos\varphi
+\frac{C}{2}\sigma^2;
\end{split}
\end{equation}
the differentiability of $\underline{W}_{\text{eff}}^h$ with respect to the variables $(\varphi,z)$ then implies the validity of \eqref{commuted_2conf}, as an immediate computation reveals.
\end{proof}
 
Proposition~\ref{prop_comm_2conf} shows, in particular, that the variational structure is preserved in the limiting model. Indeed, the equilibrium equations of the limiting system can be obtained in two equivalent ways: either by taking the limit as $h\to0$ in the equilibrium equations of the $h$-dependent system, or directly as the Euler--Lagrange equations associated with the limiting effective renormalized energy $\underline{W}_{\text{eff}}$\,. Thus, equilibrium configurations of the limiting system can be characterized as critical points of $\underline{W}_{\text{eff}}$\,, with stable equilibria corresponding to its local minima.
A direct consequence of~\eqref{commuted_2conf} is that also the force~$f^h$ and torque~$t^h$ pass to the limit as $h\to0$ and 
\begin{equation}\label{eq_grad_Wh0_2conf}
\begin{pmatrix}
    t(\varphi,z;s,\sigma)\\
    f(\varphi,z;s,\sigma)
\end{pmatrix}=
\lim_{h\to0} \begin{pmatrix}
    t^h(\varphi,z;s,\sigma)\\
    f^h(\varphi,z;s,\sigma)
\end{pmatrix}=-4Cs\begin{pmatrix}
            -z\bigl( sz\cos\varphi+\sigma \log z\bigr) \sin\varphi \\
            sz\bigg(\cos^2\varphi -\frac{1}{1-z^2}\biggr)+\sigma(\log z+1)\cos\varphi
        \end{pmatrix}.
\end{equation}
The equilibria for the effective configuration with 
the isolated disclination at the origin and one edge dislocation at $z\in(0,1)$ are found by equating to zero both expressions in \eqref{eq_grad_Wh0_2conf}.
For future use, we compute there the Hessian matrix of $\underline{W}_{\text{eff}}$
\begin{equation}\label{eq_hessian_W_eff_2conf}
\begin{split}
&\,\nabla^2_{(\varphi,z)}\underline{W}_{\text{eff}}(\varphi,z;s,\sigma) \\
=&\, 4Cs
\begin{pmatrix}
-z(sz\cos(2\varphi)+\sigma\cos\varphi \log z) & -(sz\sin(2\varphi)+\sigma\sin\varphi(\log z+1)) \\[2mm]
-(sz\sin(2\varphi)+\sigma\sin\varphi(\log z+1)) & \displaystyle  s\cos^2\varphi +\frac\sigma{z}\cos\varphi -s\frac{1+z^2}{(1-z^2)^2}
\end{pmatrix}.
\end{split}
\end{equation}

\section{Equilibrium configurations and dynamics}\label{sec_second_config}

The torque~$t$ in~\eqref{eq_grad_Wh0_2conf} vanishes, for $z\in(0,1)$, at $\varphi=0,\pi$, and the fact that $\partial^2_{\varphi\varphi} \underline{W}_{\text{eff}}(\pi,z;s,\sigma)=4Cs(-sz^2+\sigma z\log z)<0$ classifies the configuration with $\varphi=\pi$ as an unstable one; thus, from now on, we consider only $\varphi=0$ (corresponding to the $+-+$ arrangement of disclinations at finite $h$).

We notice that the torque may exhibit additional zeros: they correspond to those of
\begin{equation}\label{w12_nuovo}
\lim_{h\to0} w^h(\varphi,z;s,\sigma)=-4(s z \cos\varphi+\sigma\log z)\eqqcolon w(\varphi,z;s,\sigma),
\end{equation}
which, up to a factor, is the term the parenthesis in the first line in~\eqref{eq_grad_Wh0_2conf}.
The discussion of these additional equilibrium points, which are unstable, is carried out in Appendix~\ref{zeroes_w12_second_conf}.

Substituting $\varphi=0$ in \eqref{eq_grad_Wh0_2conf} and \eqref{eq_hessian_W_eff_2conf}, we get $t(0,z;s,\sigma)=0$, 
\begin{equation}\label{eq_force_Wh0_2conf_varphi0_ssigma}
f(0,z;s,\sigma)=-4Cs\biggl(-\frac{sz^3}{1-z^2}+\sigma(\log z+1)\biggr),
\end{equation}
and
\begin{equation*}
\begin{split}
\nabla^2_{(\varphi,z)}\underline{W}_{\text{eff}}(0,z;s,\sigma)
= 4Cs
\begin{pmatrix}
-z(sz+\sigma\log z) & 0 \\[2mm]
0 & \displaystyle  \frac{\sigma}{z}+s\frac{(z^2-3)z^2}{(1-z^2)^2}
\end{pmatrix}.
\end{split}
\end{equation*}
By introducing $\alpha\coloneqq s/\sigma\in(0,+\infty)$, the last two equations become

\begin{equation}\label{eq_force_Wh0_2conf_varphi0}
f^\alpha(z)
\coloneqq f(0,z;\alpha\sigma,\sigma)=-4C\alpha\sigma^2\biggl(1+\log z-\frac{\alpha z^3}{1-z^2}\biggr),
\end{equation}
and
\begin{equation}\label{eq_hessian_W_eff_2conf_alpha}
H^\alpha(z)
\coloneqq \nabla^2_{(\varphi,z)}\underline{W}_{\text{eff}}(0,z;\alpha\sigma,\sigma)
= 4C\alpha\sigma^2
\begin{pmatrix}
-z(\alpha z+\log z) & 0 \\[2mm]
0 & \displaystyle  \frac{1}{z}+\alpha\frac{(z^2-3)z^2}{(1-z^2)^2}
\end{pmatrix}.
\end{equation}

The equilibrium points for the force $f^\alpha$
depend on the value $\alpha$ of the relative strength of dislocation and disclination and on the signs of the diagonal elements of the Hessian matrix $H^\alpha$, 
see Table~\ref{2604271812}.

\subsection{Classification of equilibrium points}
In the limiting case $\alpha\to0$, corresponding to no dislocation in the system, we have, by continuity with respect to $\alpha$, $f^0(z)=0$: 
this is acceptable, since we would be left with the isolated disclination at the origin, which, by symmetry, is subject to no force.
 
For $\alpha>0$, there exists a critical value $\alpha_1>0$ such  that the equation $f^\alpha(z)=0$ has a unique solution $z_1$ (the pair $(\alpha_1,z_1)$ is determined imposing also the condition $\frac{\mathrm{d}}{\mathrm{d}z} f^\alpha(z)=0$), see Figure~\ref{fig:bifurcation_beta}-RIGHT. 
For $\alpha>\alpha_1$\,, no solution exists, whereas for $\alpha\in(0,\alpha_1)$ there exist two distinct equilibrium positions. 
Among these two solutions, the one corresponding to the larger value of $z$ is always unstable, since at that point $H^\alpha_{22}<0$.

It remains to discuss the nature of the smaller of the two solutions, which follows from the combined analysis of the signs of diagonal terms of the Hessian $H^\alpha$. 
First, note that $H^\alpha_{22}(z)>0$ for the smaller value of $z$ whenever $\alpha < \alpha_1$\,. 
Therefore, it suffices to analyze the sign of $H^\alpha_{11}$\,. 
This leads to a second critical value $\tilde{\alpha} < \alpha_1$\,. 
For $\alpha < \tilde{\alpha}$, one has $H^\alpha_{11}>0$, 
and the smaller~$z$ solution is stable. 
For $\tilde{\alpha} < \alpha < \alpha_1$, the stationary points are unstable. See Table~\ref{2604271812}.
Notice that no stationary point can exist for $z\leq e^{-1}$, this value corresponding to the limit case  $\alpha=0$ in the parenthesis in the expression of $f^\alpha$, see~\eqref{eq_force_Wh0_2conf_varphi0}.

\begin{table}[h!]
\centering
\begin{tabular}{|c|c|c c|c|}
\hline
\textbf{Range of $\alpha$} & \textbf{Stationary Points} 
& \multicolumn{2}{c|}{\textbf{Stability}} 
& \textbf{Reference} \\
 &  & $H^\alpha_{11}$\,, $H^\alpha_{22}$ & \textbf{Type} &  \textbf{(Fig.~\ref{fig:bifurcation_beta})}  \\
\hline
$0<\alpha<\tilde{\alpha}$
&$z\in(e^{-1},\tilde{z})$
& +,+
& stable& curves (c),(d)  \\
&  $z\in(\tilde{\tilde{z}},1)$
& -,-
& unstable
&  - \\
\hline
$\alpha=\tilde{\alpha}$
&$z=\tilde{z}$  
& 0,+
& unstable
& - \\
& $z=\tilde{\tilde{z}}$
& -,-
& unstable
& - \\
\hline
$\tilde{\alpha}<\alpha<\alpha_1$
&$z\in(\tilde{z},z_1)$   
& -,+
& unstable
& - \\
& $z\in(z_1,\tilde{\tilde{z}})$
& -,-
& unstable
& - \\
\hline
$\alpha=\alpha_1$
& $z=z_1$ 
& -,0
& unstable
& curve (b)   \\
\hline
$\alpha>\alpha_1$
& $\nexists$
&  
& 
& curve (a)   \\
\hline
\end{tabular}
\caption{Classification of equilibrium points. 
The approximate values are $\alpha_1\approx 1.853$ and $z_1\approx 0.485$.
When $\alpha=\tilde{\alpha}\approx 1.768$, the function $f^{\tilde{\alpha}}$ has two zeros, denoted by $\tilde{z}\approx 0.451$ and $\tilde{\tilde{z}}\approx 0.529$.}
\label{2604271812}
\end{table}

Below we show how the critical value $\tilde{\alpha}$ is determined and how the above analysis is carried out. 
Imposing simultaneously $f^\alpha(z)=0$ 
and $H^\alpha_{11}(z)=0$, 
under the constraints on $\alpha$ and $z$, 
yields the pair $(\tilde{z},\tilde{\alpha})$.
We study the sign of $H^\alpha_{11}(z)$: 
from \eqref{eq_hessian_W_eff_2conf_alpha},
\begin{eqnarray}
H^\alpha_{11}(z)
= -4C \alpha\sigma^2 z  ( \alpha z +\log z ) 
\end{eqnarray}
under the constraints $f^\alpha(z)=0$, 
$0\le\alpha<\alpha_1$\,, and $e^{-1}\le z<z_1$\,.
Thus we are lead to studying the sign of 
$$
1-\frac{\log z+1}{z^2},
\qquad\qquad
\textrm{for } 0\le\alpha<\alpha_1\,, e^{-1}\le z<z_1\,,
$$
which has a root at
\begin{equation}
\tilde{z} \approx 0.451,
\qquad
\tilde{\alpha}
= \frac{\log \tilde{z} + 1}{\tilde{z}^2}
\left(\frac{1 - \tilde{z}^2}{\tilde{z}}\right)
= \frac{1 - \tilde{z}^2}{\tilde{z}}
\approx 1.768.
\end{equation}

Observe that the function $z\mapsto 1-\frac{\log z + 1}{z^2}$ determines the sign of $H^\alpha_{11}(z)$ at the zeros of $f^\alpha$. 
In particular, $H^\alpha_{11}(z)>0$ whenever the solutions to $f^\alpha(z)=0$ satisfy $z < \tilde{z}$. 
This implies that the branch of stationary points with smaller $z$ is stable for $0 \le \alpha < \tilde{\alpha}$.

The configuration $(\alpha,z) = (\tilde{\alpha},\tilde{z})$ is unstable, since $z\mapsto H^\alpha_{11}(z)$ changes sign across $z=\tilde{z}$ (being positive for $z < \tilde{z}$ and negative for $z \in (\tilde{z},1)$).

Finally, we notice that the point $z=1$ is never a stationary point. 

\begin{figure}[h]
\centering
\includegraphics[width=0.95\textwidth]{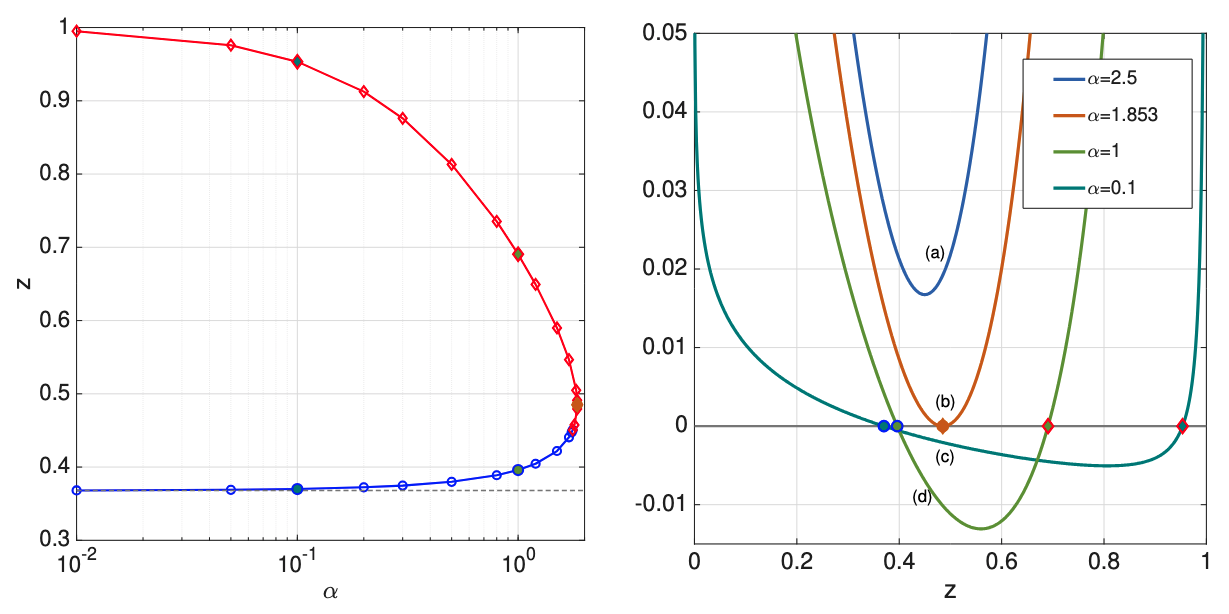}
\caption{LEFT. Bifurcation diagram of equilibrium points
as a function of parameter $\alpha$, calculated as zeros of $f^{\alpha}$
(in logarithmic scale).
RIGHT. Profiles of $z\mapsto f^\alpha(z)$ for different values of the parameter~$\alpha$. 
}\label{fig:bifurcation_beta}
\end{figure}

   \subsubsection{Discussion}

 Our analysis shows that equilibrium points may only occur in the interior of the domain, for $z\in(0,1)$, and never on the boundary. 
 We discover a family of stable equilibrium points determined by values of~$\alpha$ in the range $(0,\tilde\alpha)$, that is when the disclination is prevalent over the dislocation.
The corresponding stable branch of equilibrium points is $z\in (e^{-1},\tilde{z})$, also depending continuously on~$\alpha$. They are represented in the blue branch in Figure~\ref{fig:bifurcation_beta}. 
Here, the value $z=e^{-1}$ acts as the lower bound for the branch of these stable equilibrium points, but it is never reached, as it corresponds to the limiting case of $\alpha=0$ in the parenthesis of~\eqref{eq_force_Wh0_2conf_varphi0}.
Nevertheless, this value indicates that a characteristic length scale emerges.
It is represented by an arrangement of disclination and dislocation at mechanical equilibrium emerging from the energy minimization, and the value $z=e^{-1}$
indicates a minimum separation distance between the defects.

We observed that all stable equilibrium points of the system are achieved at $\varphi=0$, corresponding to the horizontal arrangement $+-+$ of the original system at finite $h>0$, and to the Burgers vector of the resulting dislocation pointing upwards in the limit as $h\to0$.
While many other equilibrium configurations may exist, even with  tilted orientations and non vertical alignment of the Burgers vector or horizontal alignment of the original 3-disclination system (for $h>0$), they are unstable.

\subsection{Dynamics: radial motion and collision time estimates}
 
By considering the driving force $f^\alpha$ introduced in \eqref{eq_force_Wh0_2conf_varphi0}, we can study the dissipative dynamics of the the limit system featuring the isolated disclination $\xi_3$ remaining stationary at the center of the disk and the limit dislocation placed at $z$ sliding along the horizontal segment $(0,1)$. 
This simplified one-degree-of-freedom system rules out possible rotations as $\varphi=0$ is frozen during the evolution; by restricting the motion to the radial direction, we remove rotational modes while preserving the main mechanical features of the problem. 
However this choice is not particularly restrictive as observed in \cite{BriCesMor2026}. 

The associated ODE reads 
\begin{eqnarray}\label{2605010647}
\frac{1}{\lambda}\frac{\de}{\de t} z(t)= f^\alpha(z) 
=-4C\alpha\sigma^2 \bigg( 1+\log z -\frac{\alpha z^3}{1-z^2}\bigg),
\end{eqnarray}
where $\lambda>0$ is a scalar quantity having the meaning of the inverse of a generalized viscous friction coefficient.
 
The highly nonlinear character of the ODE above makes it difficult to solve analytically. Apart from a few illustrative cases that we discuss below,
numerical solutions to \eqref{2605010647} for selected values of the parameter $\alpha$, parameterized by the value of the initial condition $z_0\coloneqq z(t=0)$, are shown in Figure~\eqref{2602211244}.
These solutions have been obtained in MATLAB environment by the explicit Euler integration method.

As a benchmark, we first consider the case of the dislocation only, whose motion is then driven by the presence of the boundary. 
To do this, we go back to the expression of the force given in \eqref{eq_force_Wh0_2conf_varphi0_ssigma}, and, by setting $\sigma=0$, and obtain the equation of motion
\begin{equation}\label{eq_dyn_2conf_sigma0}
\frac{1}{\lambda}\frac{\de}{\de t} z(t)= 4Cs^2\frac{z^3}{1-z^2},
\end{equation}
which can be solved explicitly, yielding 
\begin{equation}\label{eq_dislocation_motion}
    \log z^2(t)+\frac{1}{z^2(t)}=-8Cs^2t+\log z_0^2+\frac{1}{z_0^2}\,,
\end{equation}
$z_0\in(0,1)$ being the initial position.
By solving for $t^*$ such that $z(t^*)=1$, we find the collision time
\begin{equation}\label{eq_dislocation_collision_time}
    t^*=\frac{1}{8Cs^2}\biggl(\log z_0^2+\frac1{z_0^2}-1\biggr).
\end{equation}

Focusing on the scenario in which the dislocation is close to the boundary, $z\to1^-$, since the term $(1-z^2)^{-1}\to+\infty$, the dynamics in \eqref{2605010647} can be approximated by 
\begin{equation}
\frac{1}{\lambda}\frac{\de}{\de t} z(t) =
4C (\alpha\sigma)^2 \frac{ z^3}{1-z^2},
\end{equation}
which is exactly the same as in~\eqref{eq_dyn_2conf_sigma0}.
In both cases (be the dislocation alone in the system, of be it close to the boundary), the dynamics is as if the disclination were not present.
\footnote{It is interesting to compare the collision time obtained in~\eqref{eq_dislocation_collision_time} with that obtained for a \emph{screw} dislocation in the unit disk computed in \cite[formula (67)]{HudsonMorandotti2017}. By assuming that the initial position is $z_0=1-\delta$, we can expand~\eqref{eq_dislocation_collision_time} to second order in~$\delta$ to obtain 
\[t^*=\frac{1}{8C\pi}\biggl(\frac{2\pi}{s^2}\delta^2+O(\delta^3)\biggl),\]
which, upon taking $s^2=|b|^2=1$, coincides with the collision time $T^{\partial\Omega}_{\text{coll}}$ from \cite[formula (67)]{HudsonMorandotti2017}, up to the multiplicative constant containing the mechanical parameters.}
 
In constrast, when studying the convergence towards stable equilibrium points, the interaction between the disclination and the dislocation becomes apparent. 
To this aim, we consider the dynamics \eqref{2605010647} in a neighborhood of a point $z_\infty \in(1/e,\tilde z)$ such that $f^\alpha(z_\infty)=0$. 
By writing $z(t)=z_\infty+\zeta(t)$, for small $\zeta(t)$, and considering the first-order expansion of the force, we are lead to studying 
\begin{eqnarray}\label{eq_zeta_2conf_convtoeq}
\frac{1}{\lambda}\frac{\de}{\de t} \zeta(t)=f^\alpha(z_\infty+\zeta(t)) 
\approx \frac{\partial f^\alpha(z_\infty) 
}{\partial z}\zeta(t)=-H^\alpha_{22}(z_\infty) 
\zeta(t). 
\end{eqnarray}
Since $H^\alpha_{22}(z_\infty)>0$ owing to stability, the solution to~\eqref{eq_zeta_2conf_convtoeq} is a damped exponential, showing that the solution $t\mapsto z(t)$ converges to the equilibrium point $z_\infty$ as $t\to+\infty$.

This simple linearization analysis, together with the numerical solutions shown in Figure~\ref{2602211244}, provides a complete characterization of the dynamical behavior of the system under the simplified assumptions considered here.

First, we observe that whenever the dislocation is sufficiently close to the boundary, located at $z=1$, the boundary always attracts the dislocation with an increasing velocity. In particular, the velocity diverges as $z \to 1^{-}$, implying that the collision with the boundary occurs in finite time. Consequently, the point $z=1$ is not a stationary point of the dynamical system.

A collision with the boundary may occur for any value of the parameter~$\alpha$. More specifically, for  $\alpha \geq \tilde\alpha $, the dislocation always collides with the boundary regardless of its initial position $z_0\in(0,1)$.
In contrast, for 
$\alpha < \tilde{\alpha}$, a stable equilibrium point emerges at  $z \in (e^{-1},\tilde{z})$. In this regime, collision with the boundary is only possible if the initial condition is chosen within a sufficiently small left neighborhood of $z=1$, avoiding the basin of attraction of the stable equilibrium.

The characteristic time associated with the convergence of the dislocation toward the stable equilibrium that exists for $\alpha\in(0,\tilde{\alpha})$ is fundamentally different from the finite-time collision with the boundary discussed above. As demonstrated by the analysis leading to equation~\eqref{eq_zeta_2conf_convtoeq}, the dislocation approaches the stable equilibrium only asymptotically, reaching it only in the limit $t \to + \infty$. This behavior contrasts sharply with the finite-time collision observed when the dislocation moves toward the boundary.

This asymptotic convergence is a consequence of the elastic stress field generated by the isolated disclination and appears to be consistent with the experimental observations reported in \cite{TOKUZUMI2023118785}. In particular, those experiments provide evidence that the long-range elastic field of a disclination can significantly hinder the motion of nearby dislocations. Within the simplified framework of the present model, the disclination effectively acts as an obstacle to dislocation motion. As a result, a stable equilibrium configuration emerges, toward which the dislocation converges asymptotically, while remaining unable to pass through it.
A more comprehensive investigation is required to further characterize the disclination-induced hindrance of dislocation motion in more realistic systems with additional degrees of freedom. Such an analysis is left for future work.

\begin{figure}[htbp]
    \centering
\includegraphics[scale=0.33]{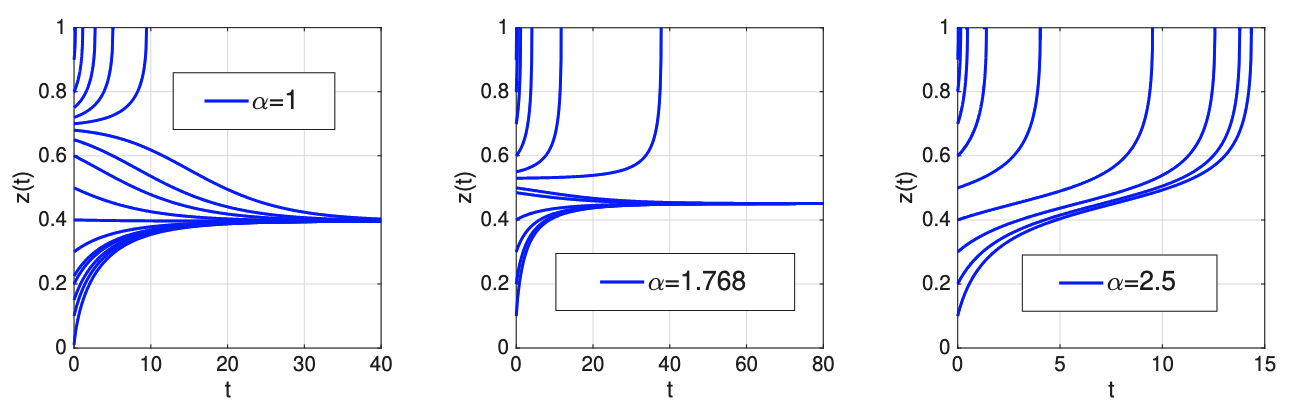}
  \caption{Solutions to~\eqref{2605010647} for selected values of~$\alpha$ and for different initial values~$z_0$\,. Solutions are obtained via numerical discretization in MATLAB environment via Euler explicit method, for the choice of parameters $\nu=0$, $E=1$, and $\sigma=1$.}
\label{2602211244}
\end{figure}

\section{Conclusion}
In this work, we have analyzed a simplified configuration of interacting disclination-disloca\-tion systems in a circular domain. Even in this reduced setting, the model reveals nontrivial features, including the emergence of stable equilibrium states inside the domain and characteristic length scales determined by the relative strength of the disclination and dislocation contributions.
Although the analysis is restricted to a single disclination interacting with a single dislocation under constrained motion, the reduced model captures the essential mechanisms of the interaction and exhibits a rich energy landscape with multiple equilibria and nontrivial stability properties. A systematic study of fully unconstrained systems involving multiple dislocations, disclinations, and dipoles is left for future work.

We have also characterized the dissipative dynamics and derived estimates for collision and relaxation times. Depending on the initial configuration and on the relative strength of the defects, the dynamics shows either finite-time boundary collisions or asymptotic convergence towards stable interior equilibria. The latter behavior appears to be a distinctive feature of coupled disclination-dislocation interactions.

Finally, our analysis uses Eshelby's representation of an edge dislocation as a disclination dipole to derive the corresponding dynamical system. An alternative approach relies on the extension of this kinematic equivalence to the dynamical level, when the renormalization is performed directly on the $h$-dependent system rather than after passing to the limiting model. This approach is explored in \cite{BriCesMor2026}.

\paragraph{Acknowledgment.}
This work was initiated when AS was a trainee student at the Joint Graduate School of Mathematics for Innovation, Kyushu University,  supported by a JASSO scholarship offered within the ``Kyushu University Program for Emerging Leaders in Science" (Q-PELS) program.
PC’s work is supported by JSPS KAKENHI Grant-in-Aid for Scientific Research (C) JP24K06797.
PC holds an honorary appointment at La Trobe University. 
PC and MM are members of the Gruppo Nazionale per l’Analisi Matematica, la Probabilità e le loro Applicazioni (GNAMPA) of the Istituto Nazionale di Alta Matematica (INdAM). 
MM acknowledges partial support from the MUR grant Geometric Analytic Methods for PDEs and Applications (2022SLTHCE cup E53D23005880006). 
This manuscript reflects only the authors’ views and opinions and the Italian Ministry cannot be considered responsible for them.
MM thanks the Institute of Mathematics for Industry, an International Joint Usage and Research Center located in Kyushu University, where part of the work contained in this paper was carried out.
The authors gratefully acknowledge the support from the 2025 IMI Joint Use Short-Term Joint Research Program (Reference No. 2025a033).

\appendix
\section{Useful results}
We collect some general results that give context to the formulae written in Section~\ref{sec_model}.

\subsection{Energetic formulation in general disks; non-dimensional form of the functional $\mathcal{I}$}\label{sec_app_nondim}
We present here the general energetic formulation in a disk of radius $R>0$.
For an atomic measure $\tilde\theta\in H^{-2}(B_R)$ describing the presence of $K\in\mathbb{N}$ disclinations,
\begin{equation}\label{atomic_measure_B_R}
\tilde\theta=\sum_{k=1}^K s_k\delta_{\tilde \xi_k}\,,  
\qquad \tilde\xi_k\in B_R\,,
\qquad \tilde \xi_k\neq\tilde \xi_j 
\quad\text{if}\quad
k\neq j
\end{equation}
($s_k$ being the Frank angles), the energy functional 
$\cI^{\tilde\theta}(\cdot;B_R)\colon H^2_0(B_R)\to\mathbb{R}$ defined for the Airy potential $\tilde v\in H^2_0(B_R)$ reads 
\begin{eqnarray}\label{eq_en_B_R}
     \mathcal{I}^{\tilde\theta}(\tilde v;B_R)= \cW(\tilde v;B_R)+\langle \tilde\theta, \tilde v\rangle =\frac{1}{2}\frac{1+\nu}{E}\int_{B_R} \big[|\nabla^2_{\tilde x} \tilde v(\tilde x)|^2-\nu(\Delta_{\tilde x} \tilde v(\tilde x))^2 \big]\,\de \tilde x+\langle \tilde \theta,\tilde v \rangle
  \end{eqnarray}
(notice that this is the functional $\mathcal{I}^\theta$ in \eqref{eq_en_B_1} when $R=1$).
By performing the change of variables 
$\tilde x=Rx$ (for $x\in B_1$), defining $x\mapsto v(x)\coloneqq \tilde v(Rx)$, and exploiting the $2$-homogeneity of the mechanical energy~$\cW$, we easily deduce the following chain of equalities 
\begin{equation}\label{2509041404}
\begin{split}
\mathcal{I}^{\tilde\theta}(\tilde v;B_R)=&\frac{1}{2}\frac{1+\nu}{E}\frac1{R^2} \int_{B_1} \big[ |\nabla^2 v(x)|^2-\nu(\Delta v(x))^2 \big]\,\de x+\langle \tilde\theta(R\,\cdot), \tilde v(R\,\cdot) \rangle \\
=&  \frac{1}{R^2} \big(\cW(v;B_1) +\langle R^2\tilde\theta(R\,
\cdot),v \rangle \big) =\frac{1}{R^2} \cI^{R^2\tilde \theta(R\,\cdot)}(v;B_1) \eqqcolon\frac{1}{R^2} \mathcal{I}^{\theta}(v; B_1),
 \end{split}
\end{equation}
where we have defined $\theta(\cdot) \coloneqq  R^2\tilde\theta(R\,\cdot)$\,. 
In particular, for the atomic measure $\delta_{\tilde\xi}$\,, we have $\tilde\xi=R\xi$ (for $\xi\in B_1$) and $\theta(x)=R^2\delta_{R\xi}(Rx)$.

\begin{remark}
    We stress here that the presence of both the second-order gradient in the mechanical energy~$\cW$ and the linear term  $\tilde v\mapsto\langle\tilde\theta,\tilde v\rangle$ destroy the scale-invariance which is typical of two-dimensional problems depending only on the first-order gradient.
    This will reflect in the evaluation of the minimal energy, see Remark~\ref{rem_min_en_eval} below.
\end{remark}

When treating an arrangement of disclinations that includes dipoles $\tilde d\pm\frac{\tilde h}2w_\varphi$ centered at $\tilde d\in B_R$\, of length $\tilde h>0$, and oriented in the direction $w_\varphi=(\cos\varphi,\sin\varphi)$, we trace the dependence on $\tilde h$ by denoting by $\tilde\theta^{\tilde h}$ the corresponding measure.
In this case, the measure $\tilde\theta^{\tilde h}$ associated with one such a dipole is 
\begin{equation}
\tilde\theta^{\tilde h}(\tilde x)=s\Big(\delta_{\tilde d+\frac{\tilde h}2w_\varphi}(\tilde x)-\delta_{\tilde d-\frac{\tilde h}2w_\varphi}(\tilde x)\Big),
\end{equation}
where $s$ is a given Frank angle. 
Consequently, if $\tilde h=Rh$ and $\tilde d=Rd$ (for $d\in B_1$), we have
\begin{equation}\label{eq_202512071033}
\begin{split}
\theta^{h}(x)=R^2 \tilde\theta^{Rh}(Rx)=& \, 
R^2s\Big(\delta_{R(d+\frac{h}2w_\varphi)}(Rx)-\delta_{R(d-\frac{h}2w_\varphi)}(Rx)\Big).
\end{split}
\end{equation}
Thanks to \eqref{2509041404} and \eqref{eq_202512071033}, if we assume the scaling as in \eqref{2509140958}
\begin{equation}\label{}
v^{h}=hv
\qquad\text{and}\qquad
\tilde v^{h}=h\tilde v,
\end{equation}
we have the following chain of equalities 
\begin{equation}\label{202509121117_bis}
\begin{split}
\cI^{\tilde\theta^{\tilde h}}(\tilde v^{h};B_R) =&\, h^2 \cI^{\tilde\theta^{Rh}/h}(\tilde v;B_R) = \frac{h^2}{R^2} \cI^{R^2\tilde\theta^{Rh}(R\,\cdot)/h}(v;B_1) 
= \frac{h^2}{R^2} \cI^{\theta^{h}/h}(v;B_1)=
\frac{1}{R^2} \cI^{\theta^{h}}(v^{h};B_1).
\end{split}
\end{equation}

\subsection{Green's function for the clamped disk problem}\label{ssec_Green}
In this section, we recall the Green's function for the clamped disk problem, since it will be used to write explicitly the minimizer of the functional \eqref{eq_en_B_R} in $H^2_0(B_R)$.

It is known (see, \emph{e.g.}, \cite[Section 4]{Nakai1978} and \cite[formula (2.9)]{CGMP2025}) that the solution to the differential problem
\begin{equation}\label{2510290855}
\begin{cases}
\displaystyle \Delta_{\tilde x}^2 u = -\frac{E}{1-\nu^2}\,s\,\delta_{\tilde\xi} & \text{in } B_R\,, \\[2mm]
u = \partial_n u =  0 & \text{on } \partial B_R\,,
\end{cases}
\end{equation}
is given, for any $\tilde x\in B_R$, by 
\begin{equation}\label{2510290857}
G_{\tilde\xi,R}(\tilde x)=\begin{cases}
    -CsR^2\,\overline{G}_{\tilde\xi,R}(\tilde x) & \text{if $\tilde x\neq\tilde \xi$,} \\[3pt]
    \displaystyle -CsR^2\,\bigg(1-\frac{|\tilde\xi|^2}{R^2}\bigg)^2 & \text{if $\tilde x=\tilde\xi$,}
\end{cases}
\end{equation}
where 
\begin{equation}\label{2510290905}
C\coloneqq \frac{E}{16\pi(1-\nu^2)}
\qquad\text{and}\qquad
\overline{G}_{\tilde\xi,R}(\tilde x)\coloneqq \frac{\abs{\tilde x-\tilde \xi}^2}{R^2} \bigg[\log\frac{\abs{\tilde x-\tilde \xi}^2}{R^2} -1 -\log(\ell_{\tilde \xi,R}(\tilde x))\bigg]+ \ell_{\tilde \xi,R}(\tilde x),
\end{equation}
with
\begin{equation}\label{elle}
\ell_{\tilde \xi,R}(\tilde x)\coloneqq \bigg(1-\frac{\abs{\tilde x}^2}{R^2}\bigg)\bigg(1-\frac{\abs{\tilde \xi}^2}{R^2}\bigg)+\frac{\abs{\tilde x-\tilde \xi}^2}{R^2}.
\end{equation}
\begin{remark}\label{rem_scale-inv}
Notice that $\ell$, and therefore $\overline{G}$, are scale-invariant, that is $\ell_{\tilde\xi,R}(\tilde x)=\ell_{\xi,1}(x)$ and $\overline{G}_{\tilde\xi,R}(\tilde x)=\overline{G}_{\xi,1}(x)\eqqcolon \overline{G}_{\xi}(x)$, if $\tilde\xi=R\xi$ and $\tilde x=Rx$ ($x,\xi\in B_1$).
Therefore, the dependence of the Green's function $G_{\tilde\xi,R}$ in \eqref{2510290857} on the radius $R$ is only through the factor $R^2$.
\end{remark}

\subsection{Explicit expression of the minimal energy}\label{ssec_min_en}
Since \eqref{2510290855} is the Euler--Lagrange equation of the functional \eqref{eq_en_B_R} for $\tilde \theta=\delta_{\tilde\xi}$\,, by superposition we obtain that the minimizer of $\cI^{\tilde\theta^{\tilde h}}(\cdot\,;B_R)$ for $\tilde \theta^{\tilde h}=\sum_{k=1}^K s_k\delta_{\tilde\xi^{\tilde h}_{k}}$ is 
\begin{equation}\label{2509121514_app}
\underline{\tilde v}(\tilde x)= -\big(\tilde\theta^{\tilde h}* G_{\tilde\xi, R}\big)(\tilde x)=-CR^2\sum_{k=1}^K s_k\, \overline{G}_{\tilde\xi^{\tilde h}_{k},R}(\tilde x).
\end{equation}
Thanks to Clapeyron's Theorem (see, \emph{e.g.}, \cite[Theorem 2.2]{CGMP2025}), we can now compute the mechanical energy at equilibrium 
\begin{equation}\label{2510291059_app}
\begin{split}
\underline{\widetilde W} \coloneqq \cW(\underline{\tilde v};B_R)  = & -\frac{1}{2} \big\langle -\tilde\theta^{\tilde h}, \underline{\tilde v} \big\rangle = \frac{CR^2}{2} \left\langle\sum_{k=1}^K s_k\delta_{\tilde\xi^{\tilde h}_{k}}\,, \sum_{j=1}^K s_j \overline{G}_{\tilde \xi^{\tilde h}_{j},R}\right\rangle \\
& = \frac{CR^2}{2} 
\sum_{k,j=1}^K s_ks_j\,\overline{G}_{ \xi^h_{j}}( \xi^h_{k}) = \frac{CR^2}{2} 
\cS\circ \overline{\cG}^{ h} 
=R^2 \underline{W}=R^2\cW(\underline{v};B_1)\,,
\end{split} 
\end{equation}  
where we used Remark~\ref{rem_scale-inv} and where $\underline{v}$ is the minimizer of $\cI^{\theta}(\cdot\,;B_1)$, for $\theta(\cdot)=R^2\tilde\theta(R\,\cdot)$, see below \eqref{2509041404}.
In \eqref{2510291059_app}, $\cS$ and $\overline{\cG}^{h}$ are the $K\times K$ matrices whose entries are 
\begin{equation}\label{matricesSG}
\cS_{kj}=s_ks_j 
\qquad\text{and}\qquad \overline{\cG}^{ h}_{kj}=\overline{G}_{ \xi^h_{j}}( \xi^h_{k}).
\end{equation}

\begin{remark}\label{rem_min_en_eval}
The equality $\widetilde{\underline{W}}=R^2\underline{W}$ (see \eqref{2510291059_app}) involving the minimal energy on the unit disk $B_1$ and the disk $B_R$ of radius $R$ recovers the usual scaling by~$R^2$ of the mechanical energy at equilibrium in a disk containing disclinations, see~\cite{SN88}. 
\end{remark}

\subsection{Additional stationary points: zeros of $w$ in \eqref{w12_nuovo}}\label{zeroes_w12_second_conf}
In this Appendix, we study the presence of an additional equilibrium point in the second configurations discussed in Section~\ref{sec_second_config} and its stability.
By relying on the parameter $\alpha=s/\sigma\in(0,+\infty)$ that measures the mutual intensity of the dislocation and disclination, we study the zeros of~\eqref{eq_grad_Wh0_2conf} coming from those of the torque given by the vanishing of the function $w$ defined in~\eqref{w12_nuovo}.
In Proposition~\ref{prop_zero_w12}, we will prove that there exists an equilibrium point for the configuration studied in Section~\ref{sec_second_config}; its instability is proved in Proposition~\ref{prop_zero_w12_unstable}.

\begin{proposition}\label{prop_zero_w12}
    There exists a value $\alpha_*\in(0,+\infty)$, whose approximate value is provided in~\eqref{valoriapprossimati} below, such that system~\eqref{eq_grad_Wh0_2conf} has a stationary point $(\varphi^{**}(\alpha),z^{**}(\alpha))$ for every $\alpha\in(0,\alpha_*)$.
\end{proposition}
\begin{proof}
    The equations for the stationary points of force and torque in the limit $h\to0$ are obtained from \eqref{eq_grad_Wh0_2conf} and \eqref{w12_nuovo} and read
    \begin{equation}\label{sistema_per_equilibrio}
    \begin{cases}
    \displaystyle \frac{\alpha z}{1-z^2}-\alpha z\cos^2\varphi-\cos\varphi(1+\log z)=0, \\[2mm]
    \alpha z\cos\varphi+\log z=0.
    \end{cases}
    \end{equation}
    By solving the second equation above, we find the condition
    \begin{equation}\label{constraint_coseno}
    \cos\varphi=-\frac{\log z}{\alpha z}\eqqcolon F_\alpha(z),
    \end{equation}
    which imposes the constraint that $F_\alpha(z)\in(-1,1)$, for every $\alpha>0$ and $z\in(0,1)$. 
    An immediate inspection of the function $z\mapsto F_\alpha(z)$ reveals that the image $F_\alpha((0,1))=[0,+\infty)$. 
    The smoothness, strict monotonicity of $F_\alpha$, and the fact that $F_\alpha(0^+)=+\infty$ and $F_\alpha(1^-)=0$ guarantee that there exists a unique point $z^*(\alpha)$ such that $F_\alpha(z^*(\alpha))=1$ and $F_\alpha(z)\in (0,1)$ for $z\in(z^*(\alpha), 1)$.
    This yields that for such values of $z$ the constraint \eqref{constraint_coseno} is satisfied and can be possibly inverted to find the corresponding value of $\varphi\in(0,\pi/2)$.

    By substituting \eqref{constraint_coseno} in the first equation in \eqref{sistema_per_equilibrio}, we obtain
    \begin{eqnarray}\label{2511111503}
    G_\alpha(z)\coloneqq\frac{\alpha z}{1-z^2}+\frac{\log z}{\alpha z}=0.
    \end{eqnarray}
    Notice that the function $G_\alpha\colon(0,1)\to\mathbb{R}$ is smooth for every $\alpha>0$, strictly increasing, and $G_\alpha(0^+)=-\infty$ and $G_\alpha(1^-)=+\infty$. 
    Therefore, there exists a unique point $z^{**}(\alpha)$ such that $G_\alpha(z^{**}(\alpha))=0$.

    These considerations imply that system \eqref{sistema_per_equilibrio} has a unique solution $(\varphi^{**}(\alpha),z^{**}(\alpha))$, provided that 
    \begin{equation}\label{condizionenecessaria}
    z^{**}(\alpha)\in (z^*(\alpha),1).
    \end{equation}
    In this case,
    \begin{equation}\label{varphi**}
    (0,\pi/2)\ni \varphi^{**}(\alpha)=\arccos(F_\alpha(z^{**}(\alpha))).
    \end{equation}

    We now investigate when condition \eqref{condizionenecessaria} is satisfied.
    We start by noting that the functions $\alpha\mapsto z^{*}(\alpha)$ and $\alpha\mapsto z^{**}(\alpha)$ are differentiable by the Implicit Function Theorem, and we now prove that they are also monotone decreasing.
    This can be obtained by implicitly differentiating the equations $F_\alpha(z^{*}(\alpha))=1$ and $G_\alpha(z^{**}(\alpha))=0$ with respect to $\alpha$. 
    We obtain
    \begin{eqnarray}\label{68}
    \frac{\de}{\de\alpha}z^*(\alpha)=\frac{-(z^{*}(\alpha))^2}{1-\log z^*(\alpha) }<0
    \;\;\text{and}\;\;
    \frac{\de}{\de \alpha}z^{**}(\alpha)
    =\frac{-2\alpha (z^{**}(\alpha))^3 \big[1 - (z^{**}(\alpha))^2\big]}
    {2\alpha^2 (z^{**}(\alpha))^2 + \big(1 - (z^{**}(\alpha))^2\big)^2}<0.
    \end{eqnarray}

    The crucial condition arising from \eqref{condizionenecessaria} is that $z^{**}(\alpha)> z^{*}(\alpha)$. 
    Upon numerical evaluation (see Figure~\ref{2511081735}), we see that the ordering of $z^{*}(\alpha)$ and $z^{**}(\alpha)$ changes with $\alpha$, so that we are lead to study the sign of the (at least continuous) function $\alpha\mapsto z^{**}(\alpha)-z^{*}(\alpha)$. 
    By the Intermediate Value Theorem, there exists a value $\alpha_*$ such that $z^{**}(\alpha_*)= z^{*}(\alpha_*)\eqqcolon z_*$\,.
    The values of $\alpha_*$ and $z_*$ are found by imposing that $F_{\alpha_*}(z_*)=1$ and $G_{\alpha_*}(z_*)=0$.
    Recalling \eqref{constraint_coseno} and \eqref{2511111503}, we find that $z_*$ is the unique solution in $(0,1)$ to 
    $1-z^2=-\log z$.
    Numerical approximations provide 
    \begin{equation}\label{valoriapprossimati}
    z_*=z^{*}(\alpha_*)=z^{**}(\alpha_*) \approx 0.451
    \qquad\text{and}\qquad
    \alpha_* \approx 1.768.
    \end{equation}
    The derivatives in \eqref{68} 
    computed with the values \eqref{valoriapprossimati} are
    \begin{equation}
    \frac{\de}{\de\alpha}z^{*}(\alpha) \bigg|_{\alpha=\alpha_*} \approx -0.251
    \qquad\text{and}\qquad
    \frac{\de}{\de\alpha}z^{**}(\alpha) \bigg|_{\alpha=\alpha_*} \approx -0.450,
    \end{equation}
    from which we deduce that, at least in a left neighborhood of $\alpha_*$ condition \eqref{condizionenecessaria} is verified.
    A numerical inspection actually yields that \eqref{condizionenecessaria} is verified for every $\alpha\in(0,\alpha_*)$.

    Therefore, if $\alpha\in(0,\alpha_*)$, then the values $(\varphi^{**}(\alpha),z^{**}(\alpha))$ given by $G_\alpha(z^{**}(\alpha))=0$ and by \eqref{varphi**} are the desired solution to~\eqref{sistema_per_equilibrio}.
\end{proof}

\begin{proposition}\label{prop_zero_w12_unstable}
    For every $\alpha\in(0,\alpha_*)$, where $\alpha_*$ is provided in the proof of Proposition~\ref{prop_zero_w12}, the stationary point $(\varphi^{**}(\alpha),z^{**}(\alpha))$ is unstable.
\end{proposition}
\begin{proof}
To prove instability we will show that the Hessian matrix~\eqref{eq_hessian_W_eff_2conf} of the energy $\underline{W}_{\text{eff}}$ computed at $(\varphi^{**}(\alpha),z^{**}(\alpha))$ is indefinite.
Indeed, evaluating \eqref{eq_hessian_W_eff_2conf} at
$(\varphi^{**}(\alpha),z^{**}(\alpha))$, therefore keeping \eqref{constraint_coseno} and \eqref{2511111503} into account, 
yields that 
\begin{equation*}\label{Hessian_second_config}
\begin{split}
H^\alpha(\varphi^{**}(\alpha),z^{**}(\alpha))\coloneqq
& \nabla^2\underline{W}_{\text{eff}}(\varphi^{**}(\alpha),z^{**}(\alpha);\alpha\sigma,\sigma) \\
=& -4C\alpha\sigma^2\!
\begin{pmatrix}
-\alpha (z^{**}(\alpha))^2\sin^2\varphi^{**}(\alpha) & \sin\varphi^{**}(\alpha)(1-\log z^{**}(\alpha)) \\
\sin\varphi^{**}(\alpha)(1-\log z^{**}(\alpha)) & \displaystyle \frac{\alpha(2-\alpha^2)z^{**}(\alpha)}{(1-(z^{**}(\alpha))^2)^2}
\end{pmatrix}\!.
\end{split}
\end{equation*}
From this 
we see that $H^\alpha_{11}(\varphi^{**}(\alpha),z^{**}(\alpha)) 
=4C\alpha^2\sigma^2(z^{**}(\alpha))^2\sin^2\varphi^{**}(\alpha)>0$ and that 
$$\det H^\alpha(\varphi^{**}(\alpha),z^{**}(\alpha))
=-16C^2\alpha^2\sigma^4\sin^2\varphi^{**}(\alpha) \frac{(1-(z^{**}(\alpha))^2)^2+2\alpha^2(z^{**}(\alpha))^2}{(1-(z^{**}(\alpha))^2)^2}<0,$$
which classifies the point $(\varphi^{**}(\alpha),z^{**}(\alpha))$ as a saddle point, and therefore an unstable equilibrium.
\end{proof}

\begin{figure}[htbp]
    \centering
    \includegraphics[scale=0.45]{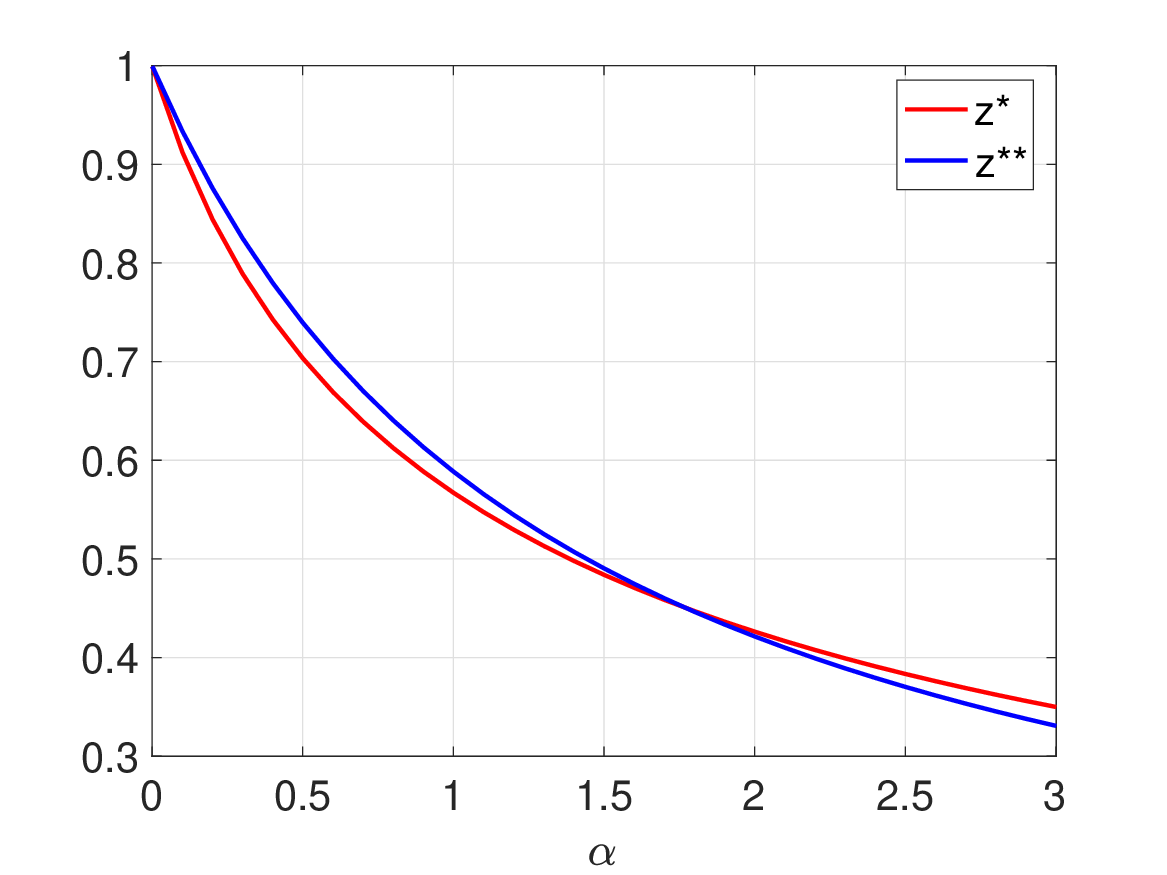}
\caption{Plots of $z^*,z^{**}$, solutions of $F_\alpha(z)=1$ (see~\eqref{constraint_coseno}) and $G_\alpha(z)=0$ (see~\eqref{2511111503}), respectively, as functions of $\alpha$.
The numerically estimated intersection occurs at $\alpha_*=1.768$, see~\eqref{valoriapprossimati}.}
\label{2511081735}
\end{figure}

\addcontentsline{toc}{section}{References}
\bibliographystyle{siam}
\bibliography{refsPatrick.bib}

@article{LAZAR06,
	author = {Markus Lazar and G{\'e}rard A. Maugin and Elias C. Aifantis},
	doi = {https://doi.org/10.1016/j.ijsolstr.2005.07.005},
	issn = {0020-7683},
	journal = {International Journal of Solids and Structures},
	number = {6},
	pages = {1787-1817},
	title = {Dislocations in second strain gradient elasticity},
	url = {https://www.sciencedirect.com/science/article/pii/S0020768305004373},
	volume = {43},
	year = {2006}}

@article{BriCesMor2026,
	author = {Nicol\`{o} Briatico and Pierluigi Cesana and Marco Morandotti},
	journal = {\href{http://arxiv.org/abs/2608.12233}{http://arxiv.org/abs/2608.12233}},
	title = {{E}ffective {D}ynamics of {D}isclination {P}airs},
	year = 2026}

@article{DENG20073646,
	author = {Shouchun Deng and Jinxing Liu and Naigang Liang},
	doi = {https://doi.org/10.1016/j.ijsolstr.2006.10.011},
	issn = {0020-7683},
	journal = {International Journal of Solids and Structures},
	number = {11},
	pages = {3646-3665},
	title = {Wedge and twist disclinations in second strain gradient elasticity},
	url = {https://www.sciencedirect.com/science/article/pii/S0020768306004173},
	volume = {44},
	year = {2007}}

@article{ROMANOV19941581,
	author = {A.E. Romanov and E.C. Aifantis},
	doi = {https://doi.org/10.1016/0956-716X(94)90312-3},
	issn = {0956-716X},
	journal = {Scripta Metallurgica et Materialia},
	number = {12},
	pages = {1581-1586},
	title = {Nonuniform misfit dislocation distributions in films},
	url = {https://www.sciencedirect.com/science/article/pii/0956716X94903123},
	volume = {30},
	year = {1994}}

@article{ROMANOV1993707,
	author = {A.E Romanov and E.C Aifantis},
	doi = {https://doi.org/10.1016/0956-716X(93)90423-P},
	issn = {0956-716X},
	journal = {Scripta Metallurgica et Materialia},
	number = {5},
	pages = {707-712},
	title = {On the kinetic and diffusional nature of linear defects},
	url = {https://www.sciencedirect.com/science/article/pii/0956716X9390423P},
	volume = {29},
	year = {1993}}

@article{Wang2023,
	address = {School of Mechanical and Power Engineering, East China University of Science and Technology, Shanghai, China.; Department of Mechanical Engineering, University of Alberta, Edmonton, AB, Canada.},
	auid = {ORCID: 0000-0003-4741-0165},
	author = {Wang, Xu and Schiavone, Peter},
	cois = {The author(s) declared no potential conflicts of interest with respect to the research, authorship and/or publication of this article.},
	copyright = {{\copyright}The Author(s) 2023.},
	crdt = {2023/11/16 04:24},
	date = {2023 Nov},
	dep = {20230426},
	doi = {10.1177/10812865231166081},
	edat = {2023/11/16 06:45},
	issn = {1081-2865 (Print); 1741-3028 (Electronic); 1081-2865 (Linking)},
	jid = {101585755},
	journal = {Math Mech Solids},
	jt = {Mathematics and mechanics of solids : MMS},
	language = {eng},
	lid = {10.1177/10812865231166081 {$[$}doi{$]$}},
	lr = {20231116},
	mhda = {2023/11/16 06:46},
	month = {Nov},
	number = {11},
	oto = {NOTNLM},
	own = {NLM},
	pages = {2396--2403},
	phst = {2023/01/09 00:00 {$[$}received{$]$}; 2023/03/11 00:00 {$[$}accepted{$]$}; 2023/11/16 06:46 {$[$}medline{$]$}; 2023/11/16 06:45 {$[$}pubmed{$]$}; 2023/11/16 04:24 {$[$}entrez{$]$}; 2023/11/08 00:00 {$[$}pmc-release{$]$}},
	pii = {10.1177{\_}10812865231166081},
	pl = {United States},
	pmc = {PMC10630138},
	pmcr = {2023/11/08},
	pmid = {37969747},
	pst = {ppublish},
	pt = {Journal Article},
	status = {PubMed-not-MEDLINE},
	title = {Interaction between an edge dislocation and a circular elastic inhomogeneity with Steigmann-Ogden interface.},
	volume = {28},
	year = {2023}}

@article{cryst9110584,
	article-number = {584},
	author = {Zhang, Zhibo and Shao, Cancan and Wang, Shuncheng and Luo, Xing and Zheng, Kaihong and Urbassek, Herbert M.},
	doi = {10.3390/cryst9110584},
	issn = {2073-4352},
	journal = {Crystals},
	number = {11},
	title = {Interaction of Dislocations and Interfaces in Crystalline Heterostructures: A Review of Atomistic Studies},
	url = {https://www.mdpi.com/2073-4352/9/11/584},
	volume = {9},
	year = {2019}}

@article{ACHARYA20171,
	author = {Amit Acharya and Michael Widom},
	doi = {https://doi.org/10.1016/j.jmps.2017.03.014},
	issn = {0022-5096},
	journal = {Journal of the Mechanics and Physics of Solids},
	pages = {1-11},
	title = {A microscopic continuum model for defect dynamics in metallic glasses},
	url = {https://www.sciencedirect.com/science/article/pii/S002250961630151X},
	volume = {104},
	year = {2017}}

@article{tng,
	author = {Tang, Yizhe},
	da = {2018/01/09},
	doi = {10.1038/s41598-017-18254-5},
	id = {Tang2018},
	isbn = {2045-2322},
	journal = {Scientific Reports},
	number = {1},
	pages = {140},
	title = {Uncovering the inertia of dislocation motion and negative mechanical response in crystals},
	ty = {JOUR},
	url = {https://doi.org/10.1038/s41598-017-18254-5},
	volume = {8},
	year = {2018}}

@article{Colin2025Disclination,
	author = {J{\'e}r{\^o}me Colin},
	doi = {10.1177/10812865241304681},
	journal = {Mathematics and Mechanics of Solids},
	number = {10},
	title = {Disclination dipole-dislocation interaction in the interface of an elastically heterogeneous bilayer},
	url = {https://doi.org/10.1177/10812865241304681},
	volume = {30},
	year = {2025}}

@article{klem,
	author = {Kleman, M. and Friedel, J.},
	doi = {10.1103/RevModPhys.80.61},
	issue = {1},
	journal = {Rev. Mod. Phys.},
	month = {Jan},
	numpages = {0},
	pages = {61--115},
	publisher = {American Physical Society},
	title = {Disclinations, dislocations, and continuous defects: A reappraisal},
	url = {https://link.aps.org/doi/10.1103/RevModPhys.80.61},
	volume = {80},
	year = {2008}}

@article{met10111517,
	article-number = {1517},
	author = {Fressengeas, Claude and Taupin, Vincent},
	doi = {10.3390/met10111517},
	issn = {2075-4701},
	journal = {Metals},
	number = {11},
	title = {Revisiting the Application of Field Dislocation and Disclination Mechanics to Grain Boundaries},
	url = {https://www.mdpi.com/2075-4701/10/11/1517},
	volume = {10},
	year = {2020}}

@article{murayama02,
	author = {M. Murayama and J. M. Howe and H. Hidaka and S. Takaki},
	doi = {10.1126/science.1067430},
	eprint = {https://www.science.org/doi/pdf/10.1126/science.1067430},
	journal = {Science},
	number = {5564},
	pages = {2433-2435},
	title = {Atomic-Level Observation of Disclination Dipoles in Mechanically Milled, Nanocrystalline Fe},
	url = {https://www.science.org/doi/abs/10.1126/science.1067430},
	volume = {295},
	year = {2002}}

@article{TOKUZUMI2020100716,
	author = {Tsubasa Tokuzumi and Shigeto Yamasaki and Wansong Li and Masatoshi Mitsuhara and Hideharu Nakashima},
	doi = {https://doi.org/10.1016/j.mtla.2020.100716},
	issn = {2589-1529},
	journal = {Materialia},
	pages = {100716},
	title = {Morphological and crystallographic features of kink bands in long-period stacking ordered Mg-Zn-Y alloy analyzed by serial sectioning SEM-EBSD observation method},
	url = {https://www.sciencedirect.com/science/article/pii/S2589152920301332},
	volume = {12},
	year = {2020}}

@article{ILTH17,
	author = {Inamura, T. and Li, M. and Tahara, M. and Hosoda, H.},
	journal = {Acta Materialia},
	pages = {351-359},
	title = {Formation process of the incompatible martensite microstructure in a beta-titanium shape memory alloy},
	volume = {124},
	year = {2017}}

@article{I19,
	author = {Inamura, T.},
	journal = {Acta Materialia},
	pages = {270-280},
	title = {Geometry of kink microstructure analysed by rank-1 connection},
	volume = {173},
	year = {2019}}

@article{BlassMorandotti17,
	author = {Blass, T. and Morandotti, M.},
	journal = {Journal of Convex Analysis},
	number = {2},
	pages = {547--570},
	title = {Renormalized energy and {P}each-{K}{\"o}hler forces for screw dislocations with antiplane shear},
	volume = {24},
	year = {2017}}

@article{SN88,
	author = {Seung, H. S. and Nelson, D. R.},
	journal = {Physical Review A},
	pages = {1005--1018},
	title = {Defects in flexible membranes with crystalline order},
	volume = {38},
	year = {1988}}

@article{RV1983,
	author = {A. E. Romanov and V. I. Vladimirov},
	journal = {physica status solidi (a)},
	pages = {11-34},
	title = {Disclinations in solids},
	volume = {78},
	year = {1983}}

@article{TOKUZUMI2023118785,
	author = {Tsubasa Tokuzumi and Masatoshi Mitsuhara and Shigeto Yamasaki and Tomonari Inamura and Toshiyuki Fujii and Hideharu Nakashima},
	issn = {1359-6454},
	journal = {Acta Materialia},
	pages = {118785},
	title = {Role of disclinations around kink bands on deformation behavior in {M}g--{Z}n--{Y} alloys with a long-period stacking ordered phase},
	volume = {248},
	year = {2023}}

@article{Taylor1934,
	author = {G. I. Taylor},
	journal = {Proceedings of the Royal Society of London. Series A},
	pages = {362-387},
	title = {The {M}echanism of {P}lastic {D}eformation of {C}rystals. {P}art {I}. {T}heoretical},
	volume = {145},
	year = {1934}}

@article{DU2025176,
	author = {Chunfeng Du and Yipeng Gao and Yizhen Li and Quan Li and Min Zha and Cheng Wang and Hailong Jia and Hui-Yuan Wang},
	doi = {https://doi.org/10.1016/j.jmst.2024.04.059},
	issn = {1005-0302},
	journal = {Journal of Materials Science and Technology},
	pages = {176-188},
	title = {A theoretical and experimental study of deformation mechanism dictated by disclination-dislocation coupling in Mg alloys at different temperatures},
	volume = {208},
	year = {2025}}

@article{ACHARYA01,
	author = {Amit Acharya},
	doi = {https://doi.org/10.1016/S0022-5096(00)00060-0},
	issn = {0022-5096},
	journal = {Journal of the Mechanics and Physics of Solids},
	number = {4},
	pages = {761-784},
	title = {A model of crystal plasticity based on the theory of continuously distributed dislocations},
	url = {https://www.sciencedirect.com/science/article/pii/S0022509600000600},
	volume = {49},
	year = {2001}}

@article{Nakai1978,
	author = {Nakai, Mitsuru and Sario, Leo},
	journal = {J. Austral. Math. Soc.},
	pages = {175-181},
	title = {Green's function of the clamped punctured disk},
	volume = {\textbf{20} (Series B)},
	year = {1978}}

@article{CDLM24,
	author = {Cesana, P. and De Luca, L. and Morandotti, M.},
	date = {2024/02/29},
	doi = {10.1137/22M1523443},
	isbn = {0036-1410},
	journal = {SIAM J. on Math. Anal.},
	journal1 = {SIAM Journal on Mathematical Analysis},
	journal2 = {SIAM J. Math. Anal.},
	n2 = {Abstract. We present a variational theory for lattice defects of rotational and translational type. We focus on finite systems of planar wedge disclinations, disclination dipoles, and edge dislocations, which we model as the solutions to minimum problems for isotropic elastic energies under the constraint of kinematic incompatibility. Operating under the assumption of planar linearized kinematics, we formulate the mechanical equilibrium problem in terms of the Airy stress function, for which we introduce a rigorous analytical formulation in the context of incompatible elasticity. Our main result entails the analysis of the energetic equivalence of systems of disclination dipoles and edge dislocations in the asymptotics of their singular limit regimes. By adopting the regularization approach via core radius, we show that, as the core radius vanishes, the asymptotic energy expansion for disclination dipoles coincides with the energy of finite systems of edge dislocations. This proves that Eshelby?s kinematic characterization of an edge dislocation in terms of a disclination dipole is exact also from the energetic standpoint.},
	pages = {79--136},
	publisher = {Society for Industrial and Applied Mathematics},
	title = {Semidiscrete {m}odeling of {S}ystems of {W}edge {D}isclinations and {E}dge {D}islocations via the {Airy} {S}tress {F}unction {M}ethod},
	type = {doi: 10.1137/22M1523443},
	volume = {56},
	year = {2024},
	year1 = {2024}}

@book{W68-BIS,
	author = {de Wit, R.},
	booktitle = {J. A. Simmons, R. de Wit, and R. Bullough (eds.) Fundamental Aspects of Dislocation Theory.},
	pages = {651-673},
	publisher = {Nat. Bur. Stand. (US), Spec. Publ. 317},
	title = {Linear {T}heory of {S}tatic {D}isclinations. Vol. $\textrm{I}$},
	year = {1970}}

@article{ZHANG18,
	author = {C. Zhang and A. Acharya and others},
	doi = {https://doi.org/10.1016/j.jmps.2018.02.004},
	issn = {0022-5096},
	journal = {J. of the Mech. and Phys. of Solids},
	pages = {258-302},
	title = {Finite element approximation of the fields of bulk and interfacial line defects},
	volume = {114},
	year = {2018}}

@article{HudsonMorandotti2017,
	author = {Hudson, Thomas and Morandotti, Marco},
	doi = {10.1137/17M1119974},
	fjournal = {SIAM Journal on Applied Mathematics},
	issn = {0036-1399,1095-712X},
	journal = {SIAM J. Appl. Math.},
	mrclass = {74H10 (37N15 74H05 82D25)},
	mrnumber = {3704273},
	mrreviewer = {Karsten\ Matthies},
	number = {5},
	pages = {1678--1705},
	title = {Properties of screw dislocation dynamics: time estimates on boundary and interior collisions},
	url = {https://doi.org/10.1137/17M1119974},
	volume = {77},
	year = {2017}}

@book{BBH1994,
	author = {Bethuel, Fabrice and Brezis, Ha\"im and H\'elein, Fr\'ed\'eric},
	doi = {10.1007/978-1-4612-0287-5},
	isbn = {0-8176-3723-0},
	mrclass = {58E20 (35Q55 49-02 58E50 82D50)},
	mrnumber = {1269538},
	pages = {xxviii+159},
	publisher = {Birkh\"auser Boston, Inc., Boston, MA},
	series = {Progress in Nonlinear Differential Equations and their Applications},
	title = {Ginzburg-{L}andau vortices},
	url = {https://doi.org/10.1007/978-1-4612-0287-5},
	volume = {13},
	year = {1994}}

@article{CGMP2025,
	author = {Cesana, Pierluigi and Grillo, Alfio and Morandotti, Marco and Pastore, Andrea},
	doi = {10.1137/24M1688096},
	fjournal = {SIAM Journal on Applied Mathematics},
	issn = {0036-1399,1095-712X},
	journal = {SIAM J. Appl. Math.},
	mrclass = {70F40 (34A60 49J10 74B99)},
	mrnumber = {4929992},
	number = {4},
	pages = {1361--1386},
	title = {Dissipative dynamics of {V}olterra disclinations},
	url = {https://doi-org.ezproxy.biblio.polito.it/10.1137/24M1688096},
	volume = {85},
	year = {2025}}

@article{ZA2018,
	author = {C. Zhang and A. Acharya},
	doi = {https://doi.org/10.1016/j.jmps.2018.06.020},
	issn = {0022-5096},
	journal = {Journal of the Mechanics and Physics of Solids},
	pages = {188-223},
	title = {On the relevance of generalized disclinations in defect mechanics},
	url = {https://www.sciencedirect.com/science/article/pii/S002250961830317X},
	volume = {119},
	year = {2018}}

@article{Polanyi1934,
	author = {M. Polanyi},
	journal = {Z. Physik},
	pages = {660-664},
	title = {{\"{U}}ber eine {A}rt {G}itterst\"{o}rung, die einen {K}ristall plastisch machen k\"{o}nnte},
	url = {https://doi.org/10.1007/BF01341481},
	volume = {89},
	year = {1934}}

@article{Orowan1934,
	author = {E. Orowan},
	journal = {Z. Physik},
	pages = {634-659},
	title = {Zur {K}ristallplastizit{\"a}t. {III}},
	url = {https://doi.org/10.1007/BF01341480},
	volume = {89},
	year = {1934}}

@article{HAGIHARA10,
	author = {K. Hagihara and N. Yokotani and Y. Umakoshi},
	doi = {https://doi.org/10.1016/j.intermet.2009.07.014},
	issn = {0966-9795},
	journal = {Intermetallics},
	number = {2},
	pages = {267-276},
	title = {Plastic deformation behavior of {M}g12{YZ}n with 18{R} long-period stacking ordered structure},
	url = {https://www.sciencedirect.com/science/article/pii/S0966979509002052},
	volume = {18},
	year = {2010}}

@article{Banhart11,
	author = {Banhart, F. and Kotakoski, J. and Krasheninnikov, A. V.},
	doi = {10.1021/nn102598m},
	eprint = {https://doi.org/10.1021/nn102598m},
	journal = {ACS Nano},
	note = {PMID: 21090760},
	number = {1},
	pages = {26-41},
	title = {Structural Defects in Graphene},
	url = {https://doi.org/10.1021/nn102598m},
	volume = {5},
	year = {2011}}

@article{IS20,
	author = {Inamura, T. and Shinohara, T.},
	journal = {Materials Transactions},
	pages = {870--874},
	title = {Rank-1 Connection of Kink Bands Formed by Non-Parallel Shears},
	volume = {61},
	year = {2020}}

@article{LAZAR05,
	author = {Markus Lazar and G{\'e}rard A. Maugin},
	doi = {https://doi.org/10.1016/j.ijengsci.2005.01.006},
	issn = {0020-7225},
	journal = {International Journal of Engineering Science},
	number = {13},
	pages = {1157-1184},
	title = {Nonsingular stress and strain fields of dislocations and disclinations in first strain gradient elasticity},
	url = {https://www.sciencedirect.com/science/article/pii/S0020722505000881},
	volume = {43},
	year = {2005}}

@article{Eshelby66,
	author = {J. D. Eshelby},
	doi = {10.1088/0508-3443/17/9/303},
	journal = {British Journal of Applied Physics},
	month = {sep},
	number = {9},
	pages = {1131--1135},
	publisher = {{IOP} Publishing},
	title = {A simple derivation of the elastic field of an edge dislocation},
	url = {https://doi.org/10.1088/0508-3443/17/9/303},
	volume = {17},
	year = 1966}

@inproceedings{acharya15,
	address = {Cham},
	author = {Acharya, A. and Fressengeas, C.},
	booktitle = {Differential Geometry and Continuum Mechanics},
	editor = {Chen, Gui-Qiang G. and Grinfeld, Michael and Knops, R. J.},
	isbn = {978-3-319-18573-6},
	pages = {123--165},
	publisher = {Springer International Publishing},
	title = {Continuum Mechanics of the Interaction of Phase Boundaries and Dislocations in Solids},
	year = {2015}}

@article{CH20,
	author = {Cesana, P. and Hambly, B.},
	journal = {Journal of Applied Probability},
	title = {A probabilistic model for interfaces in a martensitic phase transition},
	year = {2022}}

@article{BCH15,
	author = {Ball, J. and Cesana, P. and Hambly, B.},
	journal = {MATEC Web of Conferences},
	title = {A probabilistic model for martensitic avalanches},
	volume = {33},
	year = {2015}}

@article{IHM13,
	author = {Inamura, T. and Hosoda, H. and Miyazaki, S.},
	journal = {Philosophical Magazine},
	number = {6},
	pages = {618--634},
	title = {Incompatibility and preferred morphology in the self-accommodation microstructure of $\beta$-titanium shape memory alloy},
	volume = {93},
	year = {2013}}

@article{LN15,
	author = {Lei, X. W. and Nakatani, A.},
	journal = {Journal of Applied Mechanics},
	pages = {071016, 1--6},
	title = {A Deformation Mechanism for Ridge-Shaped Kink Structure in Layered Solids},
	volume = {82},
	year = {2015}}

@article{dW3,
	author = {de Wit, R.},
	journal = {Journal of Research of the National Bureau of Standards - A Physics and Chemistry},
	number = {5},
	title = {Theory of Disclinations: $\textrm{IV}$. Straight Disclinations},
	volume = {77A},
	year = {1973}}

@book{N67,
	author = {Nabarro, F. R. N.},
	publisher = {Oxford: Clarendon Press},
	series = {International Series of Monographs on Physics},
	title = {Theory of crystal dislocations},
	year = {1967}}

@article{V07,
	author = {Volterra, V.},
	journal = {Annales scientifiques de l'{\'E}cole Normale Sup{\'e}rieure},
	pages = {401-517},
	title = {Sur l'{\'e}quilibre des corps {\'e}lastiques multiplement connexes},
	volume = {24},
	year = {1907}}

@article{Acharya10,
	author = {Acharya, A.},
	journal = {Journal of the Mechanics and Physics of Solids},
	pages = {766--778},
	title = {New inroads in an old subject: plasticity, from around the atomic to the macroscopic scale},
	volume = {58(5)},
	year = {2010}}

@article{AlicandroDeLucaGarroniPonsiglione16,
	author = {Alicandro, R. and De Luca, L. and Garroni, A. and Ponsiglione, M.},
	journal = {J. Mech. Phys. Solids},
	pages = {87--104},
	title = {Dynamics of discrete screw dislocations on glide directions},
	volume = {92},
	year = {2016}}

@article{BlassFonsecaLeoniMorandotti15,
	author = {Blass, T. and Fonseca, I. and Leoni, G. and Morandotti, M.},
	journal = {SIAM Journal on Applied Mathematics},
	pages = {393--419},
	title = {Dynamics for systems of screw dislocations},
	volume = {75(2)},
	year = {2015}}

@article{CermelliGurtin99,
	author = {Cermelli, P. and Gurtin, M. E.},
	journal = {Archive for Rational Mechanics and Analysis},
	pages = {3--52},
	title = {The motion of screw dislocations in crystalline materials undergoing antiplane shear: glide, cross-slip, fine cross-slip},
	volume = {148(1)},
	year = {1999}}

@article{CermelliLeoni06,
	author = {Cermelli, P. and Leoni, G.},
	journal = {SIAM Journal on Mathematical Analysis},
	number = {4},
	pages = {1131--1160},
	publisher = {[Philadelphia] Society for Industrial and Applied Mathematics.},
	title = {Renormalized energy and forces on dislocations},
	volume = {37},
	year = {2005}}

@book{HirthLothe82,
	author = {Hirth, J. P. and Lothe, J.},
	publisher = {John Wiley \& Sons, New York},
	title = {Theory of Dislocations},
	year = {1982}}

@article{CFM2025,
	author = {Pierluigi Cesana and Edoardo Fabbrini and Marco Morandotti},
	doi = {10.1007/s10659-025-10161-5},
	issn = {1573-2681},
	journal = {Journal of Elasticity},
	number = {4},
	pages = {71},
	title = {Variational Formulation of Planar Linearized Elasticity with Incompatible Kinematics},
	url = {https://doi.org/10.1007/s10659-025-10161-5},
	volume = {157},
	year = {2025}}

\end{document}